\pdfoutput=1 
\documentclass[11pt,a4paper]{article}

\usepackage{graphicx}
\usepackage{amsmath}
\usepackage{amsfonts}
\usepackage{amssymb}
\usepackage{enumerate}
\usepackage{lscape,hyperref}
\usepackage{longtable}
\usepackage{rotating}
\usepackage{color}
\usepackage{url}
\usepackage{subfigure}
\usepackage{rotating}
\usepackage[title,titletoc,toc]{appendix}
\usepackage{mathtools}
\newtheorem{theorem}{Theorem}[section]

\usepackage[ruled,vlined,noline,linesnumbered]{algorithm2e}
\usepackage{algorithmic}
\usepackage{textcomp}
\usepackage[dvipsnames,svgnames]{xcolor}
\usepackage[normalem]{ulem}
\usepackage{subfigure}
\usepackage{tikz}
\usetikzlibrary{shapes,arrows}

\newtheorem{definition}{Definition}[section]

\newtheorem{lemma}{Lemma}[section]

\newenvironment{proof}[1][Proof]{\textbf{#1.} }{\ \rule{0.5em}{0.5em} \vspace{1ex}}

\SetCommentSty{mycommfont}
\newcommand{\esp}{\mathbb{E}}

\def\real{\mathbb{R}}

\def\doe{DoE}

\DeclareMathOperator{\EI}{EI}
\DeclareMathOperator{\argmax}{argmax}

\DeclareMathOperator{\glb}{global}
\DeclareMathOperator{\lcl}{local}

\newcommand{\cov}{\operatorname{cov}}

 \def\emailname{E-mail}%
\def\email#1{\emailname: #1}
\begin{document}

\title{TREGO: a Trust-Region Framework for Efficient Global Optimization}

\author{
Y. Diouane\thanks{Department of Mathematics and Industrial Engineering, Polytechnique Montr\'eal. \email{{\tt youssef.diouane@polymtl.ca}}
}
\and
V. Picheny \thanks{Secondmind, 72 Hills Road, Cambridge, CB2 1LA, UK.
  \email{{\tt victor@secondmind.ai}}
  }
\and
R. Le Riche \thanks{CNRS LIMOS, Mines St-Etienne and UCA, France.   \email{{\tt leriche@emse.fr}}}
  \and
  A. Scotto Di Perrotolo\thanks{ISAE-SUPAERO, Universit\'e de Toulouse, France.
              \email{{\tt alexandre.scotto-di-perrotolo@isae-supaero.fr}}
}
}

\maketitle
\footnotesep=0.4cm
\begin{abstract}
Efficient Global Optimization (EGO) is the canonical form of Bayesian optimization that has been successfully applied to solve global optimization of expensive-to-evaluate black-box problems. However, EGO struggles to scale with dimension, and offers limited theoretical guarantees.
In this work, {a trust-region framework for EGO (TREGO) is proposed and analyzed}. TREGO alternates between regular EGO steps and local steps within a trust region.
By following a classical scheme for the trust region (based on a sufficient decrease condition), {the proposed algorithm enjoys global convergence properties},
while departing from EGO only for a subset of optimization steps.
Using extensive numerical experiments based on the well-known COCO {bound constrained problems}, we first analyze the sensitivity of TREGO to its own parameters, 
then show that the resulting algorithm is consistently outperforming EGO and getting competitive with other state-of-the-art { black-box optimization} methods.
\end{abstract}

\bigskip

\begin{center}
\end{center}

\section{Introduction}

In the past 20 years, Bayesian optimization (BO) has encountered great successes and a growing popularity for solving global optimization problems with expensive-to-evaluate black-box functions. Examples range from aircraft design~\cite{forrester2007multi} to automatic  machine learning~\cite{snoek2012practical} to crop selection~\cite{picheny2017using}. In a nutshell, BO leverages non-parametric (Gaussian) processes (GPs) to provide flexible surrogate models of the objective. Sequential sampling decisions are based on the GPs, judiciously balancing exploration and exploitation in search for global optima (see~\cite{DRJones_MSchonlau_WJWelch_1998,mockus2012bayesian} for early works or~\cite{brochu2010tutorial} for a recent review).

BO typically tackles problems of the form:
\begin{eqnarray}
	\label{prob:1}
	\displaystyle \min_{x \in \Omega} f(x), 
\end{eqnarray}
 where $f$ is a pointwise observable objective function defined over a continuous set $\Omega { \subseteq} \real^n$, with $n$ relatively small (say, 2 to 20). In this work, {the objective function $f: \real^n \rightarrow \mathbb{R}$ is assumed} observable exactly (i.e., without {random} noise), bounded from below in $\real^n$ and Lipschitz continuous near appropriate limit points. 
{The constraints set $\Omega$ will be treated as explicit~\cite[i.e. not relying on estimates, as in][]{schonlau1998global} and non-relaxable~\cite{Led_Wild_2015}, meaning that the objective function cannot be evaluated outside the feasible region. In our numerical experiments, $\Omega$ will be set as a bound constraints set.}

Despite its popularity and successes, BO suffers from a couple of important drawbacks. 
First, it is very sensitive to the curse of dimensionality, as with growing dimension exploration tends to overcome exploitation 
and learning an accurate model throughout the search volume is typically not feasible within a limited number of function evaluations. 
Several recent works have tackled this problem, either making strong structural assumptions~\cite{bouhlel18,kandasamy2015high,wang2016bayesian} or 
incentivizing sampling away from the boundaries~\cite{oh2018bock,siivola2018correcting}.
Second, the theoretical properties for BO are rather limited, in particular in the noiseless context. 
For BO algorithms based on the expected improvement acquisition function, Vazquez and Bect~\cite{EVasquez_JBect_2010} showed that the sequence of evaluation points is dense in the search domain providing some strong assumptions on the objective function. 
Bull~\cite{bull2011convergence} built upon this result {to provide a convergence rate for EGO when GP models with a Mat\'ern kernel are used. 
However, the proposed convergence rate requires the addition of a well-calibrated epsilon-greedy strategy to EGO and it is valid for a limited family of objective functions.}


Over the past two decades, there has been a growing interest in deterministic Derivative-Free Optimization (DFO)~\cite{CAudet_WHare_2017,ARConn_KScheinberg_LNVicente_2009}. 
DFO methods either try to build local models of
the objective function based on samples of the function values, e.g. trust-region methods, or directly exploit a sample set of function evaluations without building an explicit
model, e.g. direct-search methods.  
Motivated by the large number of DFO applications, researchers and practitioners have made significant progress on the algorithmic and theoretical aspects (in particular, proofs of global convergence) of the DFO methods. 


In this paper, we propose to equip a classical BO method with known techniques from deterministic DFO using a trust-region scheme, and a sufficient decrease condition to accept new iterates~\cite{TGKolda_RMLewis_VTorczon_2003}. This is in line with recent propositions hybridizing BO and DFO~\cite{Turbo_2019,Trike_2016} that showed great promise empirically, but with limited theoretical guarantees. { The proposed} TREGO algorithm (Trust-Region framework for Efficient Global Optimization) benefits from both worlds: 
{TREGO rigorously achieves} global convergence under reasonable assumptions, 
while enjoying the flexible predictors and efficient exploration-exploitation trade-off provided by the GPs. Contrary to the aforementioned propositions, TREGO maintains a global search step, ensuring that the algorithm can escape local optima and maintain the asymptotic properties of BO~\cite{bull2011convergence,EVasquez_JBect_2010}.

The remainder of this article is organized as follows. Section \ref{sec:EGO} presents the classical BO framework. Section \ref{sec:TREGO} describes our hybrid algorithm, and Section \ref{sec:analysis} its convergence properties. 
Intensive numerical experiments have been carried out using the COCO test bed~\cite{hansen2021coco}. They represent months of CPU time and have allowed to study TREGO and compare it with state-of-the-art alternatives. These experiments are reported in Section \ref{sec:experiments}.
Conclusions and perspectives are finally provided in Section~\ref{sec:conc}.  
{ By default this paper uses $\ell_2$ norms.}

\section{The Efficient Global Optimization Framework}
\label{sec:EGO}

Efficient Global Optimization (EGO)~\cite{DRJones_MSchonlau_WJWelch_1998} is a class 
 of BO methods relying on two key ingredients: (i) the construction of a GP surrogate model of the objective function and (ii) the use of an  acquisition function.
 EGO proceeds along the following steps:
\begin{enumerate}
 \item an initial set of evaluations (often referred to as Design of Experiment, $\doe$) of the objective function is obtained, typically using a space-filling design~\cite{fang2005design};
 \item a GP surrogate model is trained on this data;
 \item a fast-to-evaluate acquisition function, defined with the GP model, is maximized over $\Omega$;
 \item the objective function is evaluated at the acquisition maximizer;
 \item this new observation is added to the training set and the model is re-trained;
 \item {S}teps 3 to 5 are repeated until convergence or budget exhaustion.
\end{enumerate}
{
The surrogate model is built by assuming that $f$ is a realization of a Gaussian process (GP) 
$(Y_x)_{x \in \Omega} \sim \mathcal{GP}\left(m, c \right)$, 
with prior mean function  
$m(x):= \esp(Y_x)$ and  
covariance function
$c(x, x'):= \cov(Y_x, Y_{x'})$, $x, x' \in \Omega$.
Given a DoE of size $t\in \mathbb{N}^*$, i.e., 
$\mathcal{D}_t = \{x_1,x_2, \ldots, x_t \}$ and $\mathcal{Y}_t = \{f(x_1),f(x_2), \ldots, f(x_t ) \}$, 
the posterior distribution of the process conditioned by $\mathcal{D}_t,\mathcal{Y}_t$ is Gaussian with mean and covariance given by~\cite{rasmussen2006gaussian}:
\begin{eqnarray*}
m_t(x) & :=& m(x) + \lambda_t(x) \left(Y_t - M_t \right),  \\
c_t(x, x') & :=&  c(x, x') - \lambda_t(x)c_t(x'),
\end{eqnarray*}
where $\lambda_t(x) := c_t(x)^\top C_t^{-1}$, $c_t(x) := (c(x, x_1), c(x, x_2), \dots, c(x, x_t ))^\top$, $C_t := (c(x_i, x_j) )_{1 \leq i,j \leq t}$,
$Y_t:= (f(x_1),f(x_2), \ldots, f(x_t ))^\top$ and $M_t:=(m(x_1),m(x_2), \ldots, m(x_t ))^\top$.
}

Typically, $m$ is taken as constant or a polynomial of small degree and $c$ belongs to a family of covariance functions such
as the Gaussian and Mat\'ern kernels, based on hypotheses about the smoothness
of {$f$}. Corresponding hyperparameters are often obtained as maximum likelihood
estimates; see for example~\cite{rasmussen2006gaussian,stein2012interpolation} for the corresponding details.

Once the surrogate model is built, an acquisition function is used to determine which point is most likely to enrich efficiently the model regarding the search for a global minimizer of the objective function $f$. 
{The expression of the acquisition function only depends on} the probabilistic surrogate model and usually integrates a trade-off between exploitation (i.e., low {$m_t(x)$}) and exploration (i.e., high $c_t(x,x)$)~\cite{frazier2018tutorial}.
In the noise-free setting, the canonical acquisition is Expected Improvement (EI)~\cite{DRJones_MSchonlau_WJWelch_1998}, { i.e.,
\begin{eqnarray*}
 \EI_t(x) &:=& (f_{\min} - m_t(x))\Phi\left( \frac{f_{\min} -m_t(x)}{\sqrt{c_t(x,x)} } \right) + \sqrt{c_t(x,x)} \phi\left( \frac{f_{\min} - m_t(x)}{\sqrt{c_t(x,x)} } \right),
\end{eqnarray*}
where $f_{\min} = \min_{1 \leq i \leq t}(f(x_i))$. The functions $\phi$ and $\Phi$ denote the probability and cumulative density functions, respectively, of the standard normal variable.
}
Note that many alternative acquisition functions have been proposed over the past 20 years, see for example~\cite{shahriari2015taking} for a recent review. We stress that while {the focus here is} on EI for simplicity, {the proposed} framework described later is not limited to EI and other acquisitions can be used instead (see Section \ref{sec:analysis} for suitable choices).

Given $\mathcal{D}_t$ the set of observations available at iteration $k$, the next optimization iterate $x_{k+1}$ is given by 
\begin{equation}
	x_{k+1}^{\glb} \in \underset{x \in \Omega}{\argmax} ~ \alpha(x; \mathcal{D}_t).\label{eq:egostep}
\end{equation}
where $\alpha$ corresponds to the chosen acquisition function at iteration $k$ (for EGO, $ \alpha(x; \mathcal{D}_t)~=~\EI_t(x)$).

For most existing implementations of EGO, the stopping criterion relies typically on a 
maximum number of function evaluations. 
In fact, unlike gradient-based methods where
 the gradient's norm can be used as a relevant stopping criterion which ensures a first-order stationarity, derivative-free optimization algorithms have to cope with a lack of general stopping criterion and the EGO algorithm makes no exception.

\section{A Trust-Region Framework for EGO (TREGO)}\label{sec:TREGO}

 In this section, we propose a modified version of EGO where {a control parameter is included} (which depends on the decrease of the true objective function) to ensure some form of global convergence without jeopardizing the performance of the algorithm.

\subsection{The TREGO algorithm}\label{sec:tregopresentation}
 
Our methodology follows the lines of the search/poll direct-search methods~\cite{AJBooker_etal_1998, ARConn_KScheinberg_LNVicente_2009,YDiouane_SGratton_LNVicente_2015_b, AIFVaz_LNVicente_2007}.
In the context of EGO, this results in a scheme alternating between \textit{local} and \textit{global} phases. The global phase corresponds to running {one iteration} of the classical EGO algorithm over the whole design space as in Eq. \ref{eq:egostep}. 
This phase ensures an efficient global exploration and aims at identifying the neighborhood of a global minimizer. 
The local phase corresponds to running one iteration of EGO, but restricting the search to the vicinity of the current best point ({ the trust-region} $\Omega_k$, detailed hereafter), so that
\begin{equation}
    x^{\lcl}_{k+1} \in \underset{x \in \Omega_k}{\argmax} ~ \alpha(x; \mathcal{D}_t).
\end{equation}
Associated with a proper management of { the trust-region}  $\Omega_k$, this phase ensures that the algorithm converges to a stationary point.
All the trial points, whether coming from the global or from the local phase, are included in the $\doe$ to refine the surrogate model of the objective function $f$.

By default, only the global phase is used. The local one is activated when the global phase is not successful, that is when it fails to sufficiently reduce the best objective function value.
In addition, the local phase consists of a fixed number of steps (typically only one), after which the algorithm reverts to the global phase.
Consequently, the original EGO algorithm is entirely maintained over a subset of steps.

The local phase management follows two widely used techniques in the field of nonlinear optimization with and without derivatives.
First, some form of \emph{sufficient decrease condition} {is imposed} on the objective function values to declare an iteration successful. Second, we control the size of the steps taken at each iteration using a parameter $\sigma_k$ that is updated depending on the sufficient decrease condition (increased if successful, decreased otherwise). 
Given a current best point $x_k^*$, at iteration $k$, a { trust-region around $x_k^*$} is defined as
\begin{equation}
 \Omega_k := \{x \in \Omega\; {:} \;  d_{\min} \sigma_k \le \| x - x_k^* \| \le d_{\max} \sigma_k\},\label{eq:trbounds}
\end{equation}
where $d_{\min} < d_{\max}$ are any two strictly positive real values.
The inclusion in the algorithm of the bounds $d_{\min}$ and $d_{\max}$ on the definition of $\Omega_k$ is essential to our convergence analysis. 
In practice, the constant $d_{\min}$ can be chosen very small and the upper bound $d_{\max}$ can be set to a very large number. {Note that the definition of the trust-region as given in \eqref{eq:trbounds} uses the $\ell_2$ norm, however other norms can be preferred depending on the nature of the constraints set $\Omega$. For instance, if $\Omega$ contains only bound constraints, it is more practical to use the $\ell_1$ norm as we will do in our experiments.}

At each iteration of the local phase, the following sufficient decrease condition on the objective function {is imposed}:
\begin{equation}
	\label{eq:suff_decr_cond}
	f(x^{\lcl}_{k+1}) \leq f(x^*_k ) - \rho (\sigma_k),
\end{equation}
where $\rho : \real^+ \rightarrow \real^+$ is a forcing function~\cite{TGKolda_RMLewis_VTorczon_2003}, i.e., a positive continuous nondecreasing function such that $\rho(\sigma)/\sigma \rightarrow 0$ when $\sigma \downarrow 0$ (for instance, $\rho(\sigma)=\sigma^2$).
The step size parameter $\sigma_k$ is increased if the iteration is
successful, {i.e.,  $\sigma_{k+1} = \gamma  \sigma_k$ with $\gamma \in (1, +\infty)$. An iteration is declared 
successful if  the new iterate $x^{*}_{k+1}$ decreases sufficiently the objective function. In this case, the iterate $x^{*}_{k+1}$ can be updated either within the global phase ($x^{*}_{k+1}=x^{\glb}_{k+1}$) or the local one ($x^{*}_{k+1}=x^{\lcl}_{k+1}$). 
If the sufficient decrease condition \eqref{eq:suff_decr_cond} is not satisfied, the current iterate is kept unchanged ($x^{*}_{k+1}=x^{*}_{k}$) and the step size is reduced, } $\sigma_{k+1} = \beta  \sigma_k$ with $\beta \in (0,1)$. 
A classical scheme is to keep $\beta \in (0, 1)$ constant, and apply:
\begin{eqnarray}
 \sigma_{k+1} &=& \frac{{\sigma_{k}}}{\beta} \quad \text{ if the iteration is successful,}\nonumber\\
 \sigma_{k+1} &=& \beta {\sigma_{k}} \quad \text{ otherwise.}\label{eq:decrease_cond}
\end{eqnarray}

Figure \ref{fig:overview:TREGO} is a schematic illustration of the algorithm. The pseudo-code of the full algorithm is given in Appendix \ref{apendix:A}.


\begin{figure}[!ht]
\tikzstyle{block} = [rectangle, draw, text centered, rounded corners, minimum height=4em]
\tikzstyle{Endblock} = [rectangle, draw, text centered, minimum height=4em]
\tikzstyle{decision} = [diamond, minimum width=1cm, minimum height=1cm, text centered, draw=black]
\tikzstyle{arrow} = [->,thick]
\begin{center}
\begin{tikzpicture}
\node [block] (global) {\begin{minipage}{4cm} \centering \textbf{Global phase over $\Omega$}\\ (Update the DoE) \end{minipage}};
\node [rectangle, left of=global, node distance = 6cm] (origin) {};  
\draw[arrow] (origin) -- node [above] {\begin{tabular}{c}\textbf{Start from $x_0^*$} \\ $k=0$\end{tabular}} (global);
\node [block, below of=global, node distance = 2.5cm] (local) {\begin{minipage}{5cm} \centering \textbf{Local phase \\ over {the trust-region $\Omega_k$}}\\ (Update the DoE) \end{minipage}};
\draw [arrow] (global) -- node [right] {\textbf{Failure}} (local);
\node [block, below of=local, node distance = 2.5cm,xshift=-2cm] (localSuc) {\begin{minipage}{3cm} \centering {$\sigma_{k+1}= \gamma \sigma_k$} \\ Update $x^*_{k+1}$ \end{minipage}};
\node [block, right of=localSuc, anchor=west, node distance = 2.cm] (localFail) {\begin{minipage}{3cm} \centering $\sigma_{k+1}=\beta \sigma_k$\\ $x^*_{k+1}=x^*_k$ \end{minipage}};
\draw [arrow] ([xshift=-1.5cm]local.south) to [out=-90,in=90] node [left] {\textbf{Success}\mbox{~}} node [right] {$x^\text{local}_{k+1}$} (localSuc);
\draw [arrow] ([xshift=1.1cm]local.south) to [out=-90,in=90] node [right] {\textbf{~Failure}} (localFail);
\node [decision, right of=global, below of=global, anchor=west, text width=2.5cm, align=center, inner sep = -0.5ex,xshift=4.cm,yshift=-1.2cm] (stop) {\textbf{A stopping condition is satisfied}};
\node [Endblock, below of=stop, node distance = 3cm ] (return) {\begin{minipage}{3.5cm} \centering \textbf{Stop and return current iterate as solution} \end{minipage}};
\draw[arrow] ([yshift=-0.5cm]global) -| ([xshift=-0.5cm]localSuc.west)  node [near start, below] {$x^\text{global}_{k+1}$} -- (localSuc.west) ;
\coordinate [right of=localFail, node distance=2.5cm] (joinNode) {};
\draw [thick] (localSuc.south) |- ([yshift=-0.5cm]localSuc.south) -| (joinNode); 
\draw [arrow] (localFail.east) -- (joinNode) |- (stop.west) node[above, near start, rotate=90] {\textbf{Increment} $k$};
\draw [arrow] (stop.north) |- (global) node [above,midway] {\textbf{No}};
\draw [arrow] (stop.south) -- (return.north) node [right,near start] {\textbf{Yes}};
\end{tikzpicture}
\end{center}
	\caption{An overview of the TREGO framework (detailed in Algorithm \ref{alg:TREGO}). 
	\label{fig:overview:TREGO} }
\end{figure}
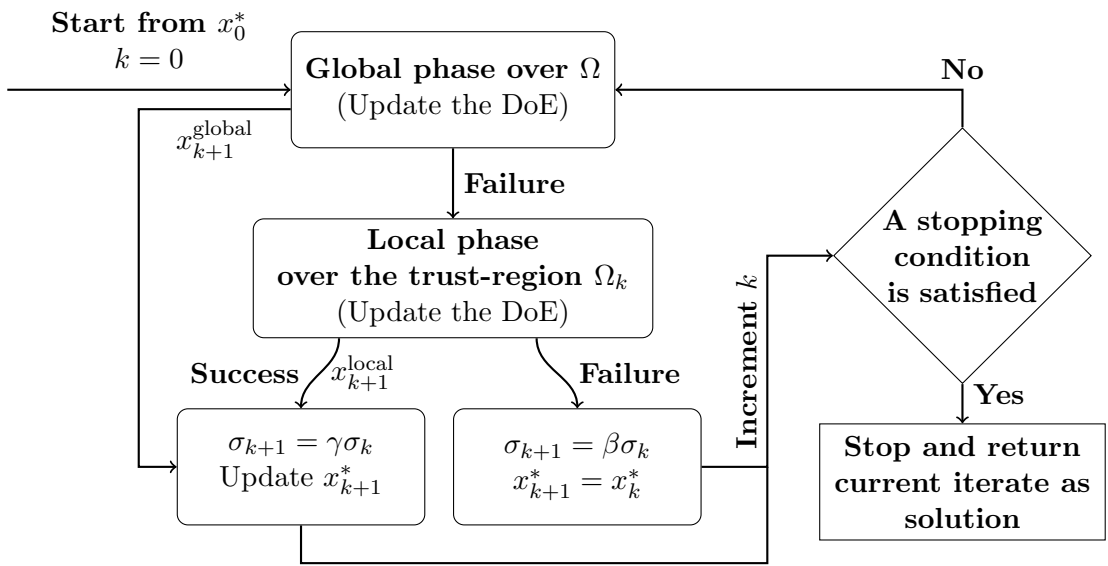

\subsection{Extensions}\label{sec:extensions}
We now present several possible extensions to TREGO. Some of these extensions are tested in the ablation study of Section \ref{sec:sensitivity}.

\paragraph{Local / global ratio:} in the previous section, a single local step is performed when the global step fails. 
The local/global ratio can easily be controlled by forcing several consecutive steps of either the global or the local phase. 
For example, a ``\texttt{gl3-5}'' (see algorithms names later) tuning would first perform three global steps regardless of their success. 
If the last step fails, it then performs five local steps. 
Such modification will not alter the structure of the algorithm. Moreover, since the convergence analysis relies on a subsequence of unsuccessful iterations, the validity of the convergence analysis (see Section \ref{sec:analysis}) is not called into question. 
In fact, during the local phase, we keep using the same sufficient decrease condition to decide whether the current iteration is successful or not. 

\paragraph{Local acquisition function:} our analysis (see Section \ref{sec:analysis}) does not require using the same acquisition for the global and local steps. For example, as EI tends to become numerically unstable in the vicinity of a cluster of observations, it might be beneficial to use the GP mean or a lower confidence bound~\cite{srinivas2010gaussian} as an acquisition function for the local step.

\paragraph{Local model:} similarly, our approach does not require using a single model for the global and local steps. One could choose a local model that uses only the points inside the trust-region to allow a better fit locally, in particular for heterogeneously varying functions. 

\paragraph{Non BO local step} finally, our analysis holds when the algorithm employed for the local step is not Bayesian. For example, using BFGS would allow a more aggressive local search, which could prove beneficial~\cite{mcleod2018optimization}. In fact, as far as the condition \eqref{eq:suff_decr_cond} {is used} to decide whether the current iteration is successful or not, the convergence theory of the next section applies.

\subsection{Related work}\label{sec:relatedwork}

\paragraph{TRIKE~\cite{Trike_2016}} (Trust-Region Implementation in Kriging-based optimization with Expected improvement) implements a trust-region-like approach where each iterate is obtained by maximizing the expected improvement acquisition function within some trust region. The two major differences with TREGO are:
1) the criterion used to monitor the step size evolution is based on the ratio {of} the expected improvement and the actual improvement, rather than sufficient decrease;
2) TRIKE does not have a global phase.
In~\cite{Trike_2016}, TRIKE is associated with a restart strategy to ensure global search.

\paragraph{TURBO~\cite{Turbo_2019}} (a TrUst-Region BO solver) carries out a collection of simultaneous BO runs using independent GP surrogate models, each within {a} different trust region. 
The trust-region radius is updated with a failure/success mechanism based on the progress made on the objective function\footnote{Importantly, TURBO uses a simple decrease rule of the objective function, which turns to be insufficient to ensure convergence to a stationary point with GP models.}. 
At each iteration, a global phase (managed by an implicit multi-armed bandit strategy) allocates samples between these local areas and thus decides which local optimizations to continue. 

Both TRIKE and TURBO display very promising performances, in particular when solving high dimensional optimization problems. 
However, both rely on several heuristics that hinder theoretical guarantees. In contrast, the use of the search/poll direct-search algorithmic design~\cite{AJBooker_etal_1998, YDiouane_SGratton_LNVicente_2015_b,ARConn_KScheinberg_LNVicente_2009, AIFVaz_LNVicente_2007} allows TREGO to benefit from global convergence properties.

\section{Convergence Analysis of TREGO}\label{sec:analysis}

Under appropriate assumptions, the global convergence of the proposed algorithm is now deduced.
 By global convergence, we mean the ability of a method to
generate a sequence of points converging to a stationary point regardless of the starting
DoE. 
A point is said to be stationary if it satisfies the {first-order} necessary conditions, in the sense that the gradient is equal to zero if the objective function is differentiable{. In the non-smooth case, the first-order necessary conditions mean that, for any direction $d$, the Clarke generalized  derivative~\cite{FHClarke_1990} along the direction $d$ is non-negative.}


{
In order to achieve our goal, the following additional assumption on the forcing function $\rho(\cdot)$ is made.
There exist constants $\bar \gamma$ and $\bar \beta$ satisfying $\bar \beta < 1 < \bar \gamma$, such that, for each $\sigma>0$, 
\begin{equation}
\rho(\beta \sigma) \ \le \ \bar \beta \rho(\sigma) ~~~~~\mbox{and} ~~~~~ \rho(\gamma \sigma) \ \le \ \bar \gamma \rho(\sigma).
\end{equation}
Such assumption is not restrictive as it holds in particular for the classical forcing functions of the form $\rho(\sigma) =c \sigma^q$ with $c>0$ and $q\ge 1$. The next lemma shows that, as far as the objective function is bounded below, the $ \sum^{+\infty}_{k=0}  \rho(\sigma_k)$ is  bounded above. The proof of the lemma is inspired from what is done in DFO~\cite{Sto_Mads_2019,EBergou_YDiouane_VKungurtsev_CWRoyer_2018a,BlanchetCartisMenickellyScheinberg19} when handling stochastic noisy estimates of the objective function, e.g.,~\cite[Theorem 1]{Sto_Mads_2019}.

\begin{lemma} \label{thm:1}
Consider TREGO
without any stopping criterion. Let $f$ be bounded below by $f_{\mathrm{low}}$. Then, one has
\begin{equation*}
	  \sum^{+\infty}_{k=0}  \rho(\sigma_k) \le 2\left( f(x^*_0) - f_{\mathrm{low}} \right)+\frac{2(1-\bar \nu)}{\bar \nu} \rho(\sigma_0) < \infty,
\end{equation*}
where  $\bar \nu=\max \left\{ \frac{\bar \gamma  -1 }{\bar \gamma -\frac{1}{2}}, \frac{1 - \bar \beta }{\frac{3}{2} -\bar \beta } \right\}$. 
\end{lemma}
\begin{proof}
For the sake of our proof, the following function is introduced
\begin{equation}
\phi_k \ 	:= \ \bar \nu (f(x^*_k)-f_{\mathrm{low}}) + (1 - \bar \nu) \rho(\sigma_k),
\label{eq:Lyapunov}
\end{equation}
where $\bar \nu=\max \left\{ \frac{\bar \gamma  -1 }{\bar \gamma -\frac{1}{2}}, \frac{1 - \bar \beta }{\frac{3}{2} -\bar \beta } \right\}$. Then, if an iteration $k$ is unsuccessful $x^*_{k+1} = x^*_k$ and $\sigma_{k+1} = \beta \sigma_k$, this leads to 
	\begin{equation} \label{eq:1:lyap}
		\phi_{k+1} - \phi_{k} = (1 - \bar\nu)(\rho(\sigma_{k+1})  - \rho(\sigma_k)) \le (1 - \bar\nu)(\bar \beta -1) \rho(\sigma_k) \le - \frac{\bar \nu}{2} \rho(\sigma_k),
	\end{equation}
	where we used $\rho(\beta \sigma_k) \le \bar \beta \rho(\sigma_k)$ and the fact that $\bar \nu \ge \frac{1 - \bar \beta }{\frac{3}{2} -\bar \beta }$.
	
	Otherwise, if the iteration $k$ is successful, then $x^*_{k+1}$ is changed and  $\sigma_{k+1} =  \gamma \sigma_k$. Then, by using the fact that $\rho(\gamma \sigma_k) \le \bar \gamma \rho(\sigma_k)$ and  $\bar \nu \ge \frac{\bar \gamma  -1 }{\bar \gamma -\frac{1}{2}}$, we obtain
	\begin{eqnarray} \label{eq:2:lyap}
		\phi_{k+1} - \phi_k &\leq& \bar \nu \rho(\sigma_k) + (1 - \bar \nu)(\bar \gamma -1) \rho(\sigma_k) \le- \frac{\bar \nu}{2} \rho(\sigma_k).
	\end{eqnarray}
Hence, from \eqref{eq:1:lyap} and \eqref{eq:2:lyap}, one deduces that for any iteration $k$, one gets
\begin{equation}
    \phi_{k+1} - \phi_{k} \ \le \ - \frac{\bar \nu}{2} \rho(\sigma_k).
\end{equation}
Thus, by applying the sum over the subscript $k$, one gets for a given iteration index $n$
\begin{equation*} 
\phi_{n+1} - \phi_0  = \sum^{n}_{k=0}  \phi_{k+1} - \phi_k =  \le - \frac{\bar \nu}{2}    \sum^{n}_{k=0}  \rho(\sigma_k).
\end{equation*}
Since $\phi_{n+1} \geq 0$, one deduces that by taking $n \to \infty$
\begin{equation*}
	  \sum^{+\infty}_{k=0}  \rho(\sigma_k) \le 2\left( f(x^*_0) - f_{\mathrm{low}} \right)\frac{2(1-\bar \nu)}{\bar \nu} \rho(\sigma_0) < \infty.
\end{equation*}
	\end{proof}

From Lemma~\ref{thm:1}, we conclude that the full sequence $\{\rho(\sigma_k)\}$ must converge to zero. Thus, by using the properties of the forcing function $\rho(.)$ (i.e., a continuous and nondecreasing function), one deduces that 
$
\lim_{k\to +\infty} \sigma_k=0.
$ The result is stated in the next theorem.

\begin{theorem} \label{cor:1}
Consider TREGO
without any stopping criterion. If the objective function $f$ is bounded below, then one has
$$\lim_{k \rightarrow +\infty} \sigma_k = 0.$$
\end{theorem}

We now introduce the following definition (similar to those in~\cite{Sto_Mads_2019,CAudet_WHare_2017,CAudet_JEDennis_2003,CAudet_JEDennis_2006}) to show the
existence of convergent subsequences of TREGO iterates.

\begin{definition}\cite[Definition 5]{Sto_Mads_2019}
 A convergent subsequence $\{x^*_k \}_{k \in \mathcal{K}}$ of TREGO iterates (for some subset
of indices $\mathcal{K}$) is said to be a refining subsequence, if and only if $\{\sigma_k \}_{k \in \mathcal{K}}$ converges
to zero. The limit $x^*$ of $\{x^*_k \}_{k \in \mathcal{K}}$ is called a refined point.
\end{definition}

Assuming that TREGO is producing iterates that lie in a compact set, one can ensure the existence of a refining subsequence.

\begin{theorem} \label{liminf2}
Consider TREGO
without any stopping criterion. Let $f$ be bounded below. If the sequence $\{x^*_k\}$  lies in a compact set, then there exists a convergent refining subsequence  $\{x^*_k \}_{k \in \mathcal{K}}$. 
\end{theorem}

The proposed convergence analysis will rely on iterates from the local phase. Thus, in what comes next, we will use $\{\hat x^{\lcl}_k\}_{k\in \mathcal{K}'} \subseteq \{ x^{*}_k\}_{k\in \mathcal{K}}$ (where $\mathcal{K}' \subseteq \mathcal{K}$ is an infinite subset of indices) to denote a refining subsequence associated with TREGO local phase iterates. The global convergence will be achieved by establishing that some type of
directional derivatives are non-negative at limit points of refining
subsequences along certain limit directions, known as refining directions (see~\cite{Sto_Mads_2019,CAudet_WHare_2017,CAudet_JEDennis_2003,CAudet_JEDennis_2006}).

\begin{definition}
Given a convergent refining subsequence (associated with the TREGO local phase) $\{\hat x^{\lcl}_k\}_{k\in \mathcal{K'}}$ and its corresponding
refined point $x^*$. Let $\{d_k\}_{k\in \mathcal{K'}}$ be a sequence such that,  $d_k := (x^{\lcl}_{k+1}-\hat x^{\lcl}_k )/\sigma_k$, for all $k\in \mathcal{K'}$. A direction $d$ is said to be a refining direction $x^*$ if and only if there exists
an infinite subset $\mathcal{L} \subseteq \mathcal{K}'$ such that $\lim_{k \in \mathcal{L}} d_k = d$. 
\end{definition}
Note that by construction, one has $d_{\min}\le d_k \le d_{\max}$, for all $k \in \mathcal{K}'$. Thus, the existence of a refining direction $d$ is justified as the sequence $\{d_k\}_{k\in \mathcal{K'}}$ lies in a compact set.
}

When $f$ is Lipschitz continuous near $x^*$, one can make use of the Clarke-Jahn generalized derivative
along a direction~$d$
\[
f^\circ(x^*;d) \; := \;
\limsup_{\begin{array}{c}x \rightarrow x^*, x\in\Omega\\t\downarrow0,x+td\in\Omega\end{array}}\frac{f(x+td)-f(x)}{t}.
\]
(Such a derivative is essentially the Clarke generalized directional derivative~\cite{FHClarke_1990},
adapted by Jahn~\cite{JJahn_1996} to the presence of constraints.)
However, for the proper definition of $f^\circ(x^*;d)$, one needs to guarantee that
$x+td\in\Omega$ for $x \in \Omega$ arbitrarily close to~$x^*$ which is
assured if $d$ is hypertangent to $\Omega$ at $x^*$. {In the following definition from~\cite{CAudet_JEDennis_2006,ARConn_KScheinberg_LNVicente_2009}, we will use
the notation
$B(x;\Delta) := \{ y \in \mathbb{R}^n: \| y-x \| < \Delta \}$ to denote the open ball of radius $\Delta$ centered at $x$.}

\begin{definition}{\cite[Definition 3.3]{CAudet_JEDennis_2006}}
A vector $d\in \real^n$ is said to be a hypertangent vector to the set $\Omega \subseteq \real^n$
at the point $x$ in $\Omega$ if there exists a scalar $\epsilon>0$ such that
\begin{equation*}
y+tw\in\Omega \quad\forall y\in\Omega\cap B(x;\epsilon),\quad w\in B(d;\epsilon) \quad\text{and}\quad 0<t<\epsilon.
\end{equation*}
\end{definition}

The hypertangent cone to $\Omega$ at $x$, denoted by~$T_\Omega^H(x)$, is the set of all hypertangent vectors
to $\Omega$ at $x$.
Then, the Clarke tangent cone to $\Omega$ at $x$ (denoted by~$T_\Omega(x)$)
can be defined as the closure of the hypertangent cone~$T_\Omega^H(x)$
(when the former is nonempty, an assumption we need to make for global convergence anyway).{
The Clarke tangent cone generalizes the notion of tangent cone in nonlinear programming~\cite{JNocedal_SJWright_2006}.  In the following  definition from~\cite{CAudet_WHare_2017,CAudet_JEDennis_2006,FHClarke_1990,ARConn_KScheinberg_LNVicente_2009}, we give the formal notion of the Clarke tangent cone.

\begin{definition}\cite[Definition 3.5]{CAudet_JEDennis_2006}
A vector $d\in \real^n$ is said to be a Clarke tangent vector to the set $\Omega \subseteq \real^n$
at the point $x$ in the closure of $\Omega$ if for every sequence $\{y_k\}$ of elements of
$\Omega$ that converges to $x$ and for every sequence of positive real numbers $\{t_k\}$
converging to zero, there exists a sequence of vectors $\{w_k\}$ converging to $d$ such
that $y_k+t_kw_k\in\Omega$, for a sufficiently large~$k$. The set $T_\Omega(x)$ of all Clarke tangent vectors to $\Omega$ at $x$ is called the Clarke tangent cone to $\Omega$ at $x$.
\end{definition}
}

{If we assume that~$f$ is Lipschitz continuous near~$x^*$ and by using~\cite[Propostion 3.5]{CAudet_JEDennis_2006},} for any a direction~$v$ in the {Clarke} tangent cone (possibly not in the hypertangent one), one can consider the
Clarke-Jahn generalized derivative to $\Omega$ at $x^*$ as the limit
\[
f^\circ(x^*;v) \; = \; \lim_{d \in T_\Omega^H(x^*), d \rightarrow v} f^\circ(x^*;d).
\]
A point~$x^* \in \Omega$ is considered Clarke stationary
if $f^\circ(x^*;d) \geq 0$, $\forall d \in T_\Omega(x^*)$.
Moreover, when $f$ is strictly differentiable at $x^*$, one has $f^\circ(x^*;d)= \nabla f(x^*)^\top d$. Hence in this case, if $x^*$ is a Clark stationary point is being equivalent to $ \nabla f(x^*)^\top d \geq 0$ , $\forall d \in T_\Omega(x^*)$.

{
It remains now to state the next lemma which will be useful for the
proof of the optimality result based on the Clarke derivative. The proof of this lemma is inspired from~\cite[Theorem 5]{Sto_Mads_2019}.

\begin{lemma}\label{lm:phi}
Consider TREGO
without any stopping criterion. Let $f$ be bounded below by $f_{\mathrm{low}}$. Then, one has
\begin{eqnarray*}
    \liminf_{k \to + \infty} \frac{f(\hat x^{\lcl}_k) - f(\hat x^{\lcl}_k + \sigma_k d_k)}{\sigma_k} & \le & 0.
\end{eqnarray*}
\end{lemma}
\begin{proof}
By contradiction, assume that there exists $\epsilon>0$ such that,
\begin{eqnarray}\label{lm:phi:eq:1}
\frac{f(\hat x^{\lcl}_k) - f(\hat x^{\lcl}_k + \sigma_k d_k)}{\sigma_k} & \ge & \epsilon,~~~~~~\mbox{for all $k \in \mathbb{N}$}.
\end{eqnarray}
From Theorem~\ref{cor:1}, one has $\lim_{k \to +\infty} \sigma_k =0$, hence by using the forcing function properties one has also $\lim_{k \to +\infty} \frac{\rho(\sigma_k)}{\sigma_k} =0$. This means that there exists $k_0>0$, such that 
\begin{eqnarray}\label{lm:phi:eq:2}
\rho(\sigma_k) & \le & \epsilon \sigma_k,~~~~~~\mbox{for all $k\ge k_0$}.
\end{eqnarray}
By combining \eqref{lm:phi:eq:1} and \eqref{lm:phi:eq:2}, one gets  
$$
f(\hat x^{\lcl}_k) - f(\hat x^{\lcl}_k + \sigma_k d_k) \ge \rho(\sigma_k), ~~~~~~\mbox{for all $k\ge k_0$}.
$$
Hence, for all $k\ge k_0$, the $k$-th iteration of TREGO is successful and $\sigma_{k+1}=\gamma \sigma_k$ (with $\gamma>1$). This contradicts $\lim_{k \to +\infty} \sigma_k =0$ and thus the claim \eqref{lm:phi:eq:1} is false.
\end{proof}
}


{
The next theorem states the global convergence of TREGO. The obtained result is in the vein of those first established in~\cite[Theorem 3.2]{CAudet_JEDennis_2006}
for simple decrease and Lipschitz continuous functions and later generalized in~\cite{YDiouane_SGratton_LNVicente_2015,LNVicente_ALCustodio_2011} for sufficient decrease
and directionally Lipschitz functions. 

\begin{theorem}
Let the assumptions made in Theorem~\ref{thm:1} hold. Let $x^* \in \Omega$ be a refined point of a refining subsequence associated with the TREGO local phase $\{\hat x^{\lcl}_k\}_{k\in \mathcal{K}'}$. Assume that~$f$ is Lipschitz continuous near~$x^*$
and that~$T_\Omega^H(x^*) \neq \emptyset$. Let $d \in T_\Omega^H(x^*)$ be a refining direction associated with $\{ d_k \}_{k \in \mathcal{K}'}$. 
Then, the Clarke-Jahn generalized
derivative of $f$ at $x^*$ in the direction $d$ is nonnegative, i.e., $f^\circ(x^*;d) \geq 0$.
\end{theorem}

\begin{proof}
In fact, from Lemma \ref{lm:phi}, there exists a subset $\mathcal{K}$ such that $\lim_{k \in \mathcal{K}}\frac{f(\hat x^{\lcl}_k) - f(x^{\lcl}_{k+1})}{\sigma_k}~\le~0$. From Theorem \ref{cor:1}, one has also $\lim_{k \in \mathcal{K}}\sigma_k=0$. Now, by using Theorem \ref{liminf2}, there exists a subset $\mathcal{K}'\subseteq \mathcal{K}$ such that $\lim_{k \in \mathcal{K}'}\hat x^{\lcl}_k=x^*$. Then, since the subsequence $\{ d_k \}_{k \in \mathcal{K}}$  lies in a compact set, there must exist a subset $\mathcal{L} \subseteq \mathcal{K}'$ such that $\{ d_k \}_{k \in \mathcal{L}}$ converges to $d$ and  
\begin{eqnarray} \label{thm:cv:eq:1}
\lim_{k \in \mathcal{L}}\frac{f(\hat x^{\lcl}_k) - f(\hat x^{\lcl}_k + \sigma_k d_k)}{\sigma_k} &\le&0.
\end{eqnarray}

From the Lipschitz continuity of~$f$ near $x^*$ and using~\cite[Proposition 3.9]{CAudet_JEDennis_2006}, one deduces that the Clarke generalized derivative is continuous with respect to $d$ on the Clarke tangent cone. Hence,
$$
f^{\circ}(x^*;d)= \lim_{k\in \mathcal{L}} f^{\circ}(x^*;d_k)
$$
Additionally, one has $\hat x^{\lcl}_k + \sigma_k d_k \in \Omega$ for all $k \in \mathcal{L}$, this leads to
\begin{eqnarray} \label{thm:cv:eq:2}
f^{\circ}(x^*;d) &=&\lim_{k\in \mathcal{L}} \limsup_{\begin{array}{c}x \rightarrow x^*, x\in\Omega\\t\downarrow0,x+td_k\in\Omega\end{array}} \frac{f(x+t d_k)-f(x)}{t} \nonumber,\\ &\geq &\limsup_{k\in \mathcal{L}} \frac{f(\hat x^{\lcl}_k+\sigma_k  d_k )-f(\hat x^{\lcl}_k )}{\sigma_k}.
\end{eqnarray}
Hence, by substituting \eqref{thm:cv:eq:1} into \eqref{thm:cv:eq:2}, one gets
\begin{eqnarray*}
f^{\circ}(x^*;d)
&\geq & \limsup_{k\in \mathcal{L}} \frac{f(\hat x^{\lcl}_k+\sigma_k d_k ) -f(\hat x^{\lcl}_k )}{\sigma_k}   \ge 0.
\end{eqnarray*}

\end{proof}
}


\section{Numerical Experiments}\label{sec:experiments}
The objective of this section is twofold: first, to evaluate the sensitivity of TREGO to its own parameters and perform an ablation study; second, to compare our algorithm with the original EGO and other BO alternatives to show its strengths and weaknesses. {TREGO is  available both in  the \texttt{R} package \texttt{DiceOptim} \footnote{\texttt{https://cran.r-project.org/package=DiceOptim}} and  python library \texttt{trieste} \footnote{\texttt{https://secondmind-labs.github.io/trieste/}}.}

\subsection{Testing procedure using the BBOB benchmark}

Our experiments are based on the COCO (COmparing Continuous Optimizers,~\cite{hansen2021coco}) software. 
COCO is a recent effort to build a testbed that allows the rigorous comparison of optimizers.
We focus here on the noiseless {Black-Box Optimization Benchmarking} (BBOB) test suite in the \textit{expensive objective function} setting~\cite{hansen2010comparing}
that contains 15 instances of 24 functions~\cite{bbob_online_functions}; each function is defined for an arbitrary number of parameters ($\geq 2$) to optimize. Each instance corresponds to a randomized modification of the original function (rotation of the coordinate system and a random translation of the optimum). The functions are divided into 5 groups: 1) separable, 2) unimodal with moderate conditioning, 3) unimodal with high conditioning, 4) multi-modal with adequate global structure, and 5) multi-modal with weak global structure. Note that group 4 is often seen as the main target for Bayesian optimization~\cite{DRJones_MSchonlau_WJWelch_1998}.
The full description of the functions is available in Appendix \ref{apendix:B} (see Table \ref{tab-bbob_functions}).

\emph{A problem} is a pair [function, target to reach]. Therefore, for each instance of a function, there are several problems to solve of difficulty varying with the target value. The \emph{Empirical Run Time Distributions} (ERTD) gives, for a given budget (i.e. number of objective function evaluations), the proportion of problems which are solved by an algorithm. This metric can be evaluated for a single function and dimension, or averaged over a set of functions (typically over one of the 5 groups or over the 24 functions).

To set the target values and more generally define a reference performance, COCO relies on a composite fake algorithm called best09. best09 is made at each optimization iteration of the best performing algorithm of the BBOB 2009~\cite{hansen2010comparing}. 
In our experiments, the targets were set at the values reached by best09 after $[0.5,1,3,5,7,10,15,20]\times n$ function evaluations.
Note that outperforming best09 is a very challenging task, as it does not correspond to the performance of a single algorithm but of the best performing algorithm for each instance. In the following, the best09 performance is added to the plots as a reference. In addition, we added the performance of a purely random search, to serve as a lower bound.



\subsection{Implementation details}
For a fair comparison, TREGO, EGO and TRIKE are implemented under a unique framework, based on the \texttt{R} packages \texttt{DiceKriging} (Gaussian process models) and \texttt{DiceOptim} (BO)~\cite{picheny2014noisy,roustant2012dicekriging}. 
Our setup aligns with current practices in BO~\cite{frazier2018tutorial,shahriari2016taking}, as we detail below.

All GP models use a constant trend and an anisotropic Mat{\'e}rn covariance kernel with smoothness parameter $\nu = 5/2$. The GP hyperparameters are inferred by maximum likelihood after each addition to the training set; the likelihood is maximized using a multi-start L-BFGS scheme. 
In case of numerical instability, a small regularization value is added to the diagonal of the covariance matrix. 

Trust regions are defined using the $\ell_1$ norm, see~\eqref{eq:trbounds}, so that they are hyper-rectangles.
This allow us to optimize the expected improvement using a multi-start L-BFGS scheme. Each experiment starts with an initial set of $2n+4$ observations, generated using latin hypercube sampling improved through a maximin criterion~\cite{fang2005design}. 
All BO methods start with the same DoEs, and the DoE is different (varying the seed) for each problem instance.

For locGP, the local model uses the same kernel and mean function as the global one, but its hyperparameters are inferred independently. To avoid numerical instability, the local model is always trained on at least $2n+1$ points. If the trust-region does not contain enough points, the points closest to the center of the trust-region are also added to the training set.


\subsection{Sensitivity analysis and ablation study}\label{sec:sensitivity}

TREGO depends on a number of parameters (see Section \ref{sec:TREGO}) and has some additional degrees of freedom worth exploring (see Section \ref{sec:extensions}). The objective of these experiments is to answer the following questions:
\begin{enumerate}
 \item is TREGO sensitive to the initial size of the trust region?
 \item is TREGO sensitive to the contraction factor $\beta$ (see Eq. \ref{eq:decrease_cond}) of the trust region?
 \item is using a local model beneficial?
 \item is there an optimal ratio  {of} global and local steps?
\end{enumerate}

To answer these questions, we run a default version of TREGO and 9 variants, as reported in Table \ref{tab-algoNames}. The contraction parameter $\beta$ is either $0.9$ (which is classical in DFO algorithms) or $0.5$ (which corresponds to an aggressive reduction of the trust region). 
{ The choice of the initial trust-region $\sigma_0$, within the default TREGO, corresponds to setting the initial trust-region volume to $20\%$ of the search space. In this case, the initial trust-region volume is given by $(2\sigma_0)^{n}$.} We test also as alternatives with a small initial trust-region (i.e., 10\% of the search space) and a larger one (i.e., 40\% of the search space). The global-local ratio varies from 10-1 (which is expected to behave almost similarly to the original EGO) to 1-10 (very local).

\begin{table}

\begin{tabular}{|p{0.2\linewidth}p{0.7\linewidth}|}
\hline
Acronym & Solvers \\
\hline  \hline
random & random search \\
best09 & best of all BBOB 2009 competitors at each budget
\cite{auger:inria-00471251} \\
TRIKE & TRIKE algorithm of~\cite{Trike_2016}\\
SMAC & SMAC algorithm of~\cite{SMAC_2011} \\
DTS-CMA & DTS-CMA algorithm of~\cite{DTS_CMA_2019} \\
EGO & original EGO algorithm of~\cite{DRJones_MSchonlau_WJWelch_1998} \\
\hline \hline
TREGO & default TREGO with $\beta=0.9$, { $\gamma=1/\beta$, $\sigma_0=\frac{1}{2}(1/5)^{1/n}$, $\rho(\sigma)=\sigma^2$, $d_{\max}=1$, $d_{\min}=10^{-6}$ }, global/local ratio = 1 / 1 (i.e., $G=1$ and $L=1$), with no local GP model\\
gl1-10, gl1-4, gl4-1 and gl10-1 & TREGO with a global/local ratio of 1/10, 1/4, 4/1 and 10/1, respectively\\
smV0 and lgV0 &  TREGO with small (i.e., $\sigma_0=\frac{1}{2}(1/10)^{1/n}$) and large  (i.e., $\sigma_0~=~\frac{1}{2}(2/5)^{1/n}$) initial trust-region size\\
fstC & TREGO with fast contraction of the trust-region, i.e., $\beta=0.5$\\
fstCsmV0 &  TREGO with fast contraction of the trust-region and small $\sigma_0$\\
locGP &  TREGO with a local GP model\\
\hline
\end{tabular}
\caption{Names of the compared algorithms. For the TREGO variants, when not specified, the parameter values are the ones of the default, TREGO.   \label{tab-algoNames}}
\end{table}

Because of the cost of a full COCO benchmark with EGO-like algorithms, the interaction between these parameters is not studied. 
Also, the ablation experiments are limited to the problems with dimensions 2 and 5 and relatively short runs ($30n$ function evaluations).
With these settings and 15 repetitions of each optimization run, an EGO algorithm is tested within a couple of days of computing time on a recent single processor.

Figure~\ref{fig-summarymerged}, top row, 
summarizes our study on the effect of the global versus local iterations ratio. There is measurable advantage of algorithms devoting more iterations to local rather than global search. gl1-4 and gl1-10 consistently outperform gl4-1 and gl10-1. gl1-4 and gl1-10 slightly outperform the TREGO baseline,
the effect being more visible with higher dimension (see also Figure~\ref{fig-summaryCompareLong} for results with 10 dimensions).

By further splitting results into function groups (see Figure~\ref{fig-CV0locGP_by_groups} in Appendix), it is observed that the performance gain due to having more local iterations happens on the unimodal function groups (the 2nd and 3rd, i.e., unimodal functions with low and high conditioning) when less difference can be observed on multimodal functions (first, fourth and fifth group). For multimodal functions with a weak global structure (fifth group, bottom right plot of Figure~\ref{fig-CV0locGP_by_groups}), gl10-1 is even on average (over the budgets) the best strategy. These findings are intuitive, as unimodal function may not benefit at all from global steps,
while on the other hand a too aggressively local strategy (e.g. gl1-10) may get trapped in a local optimum of a highly multimodal function.
Overall on this benchmark, gl1-4 offers the best trade-off over all groups between performance and robustness. 

\begin{figure}[!ht]
\centering
\begin{tabular}{|cc|}
\hline
$n=2$ & $n=5$ \\
\hline \hline
\includegraphics[width=0.36\textwidth]{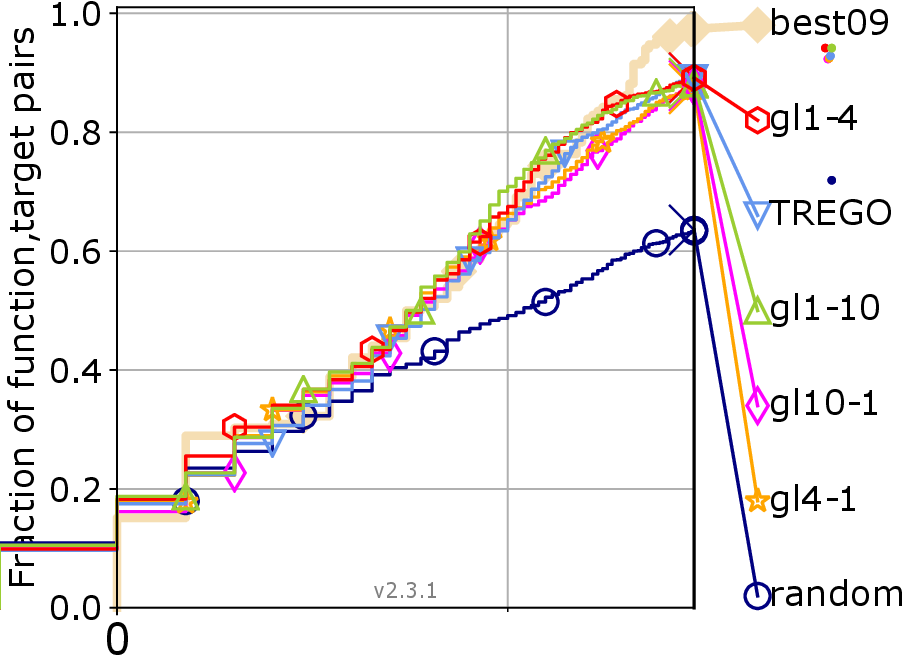}&
\includegraphics[width=0.36\textwidth]{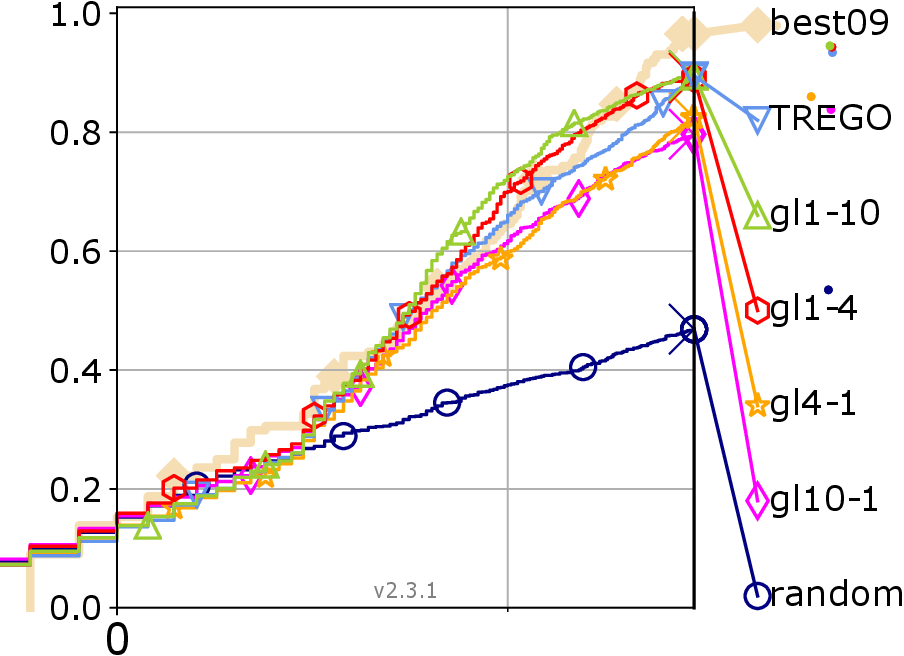}\\
\includegraphics[width=0.36\textwidth]{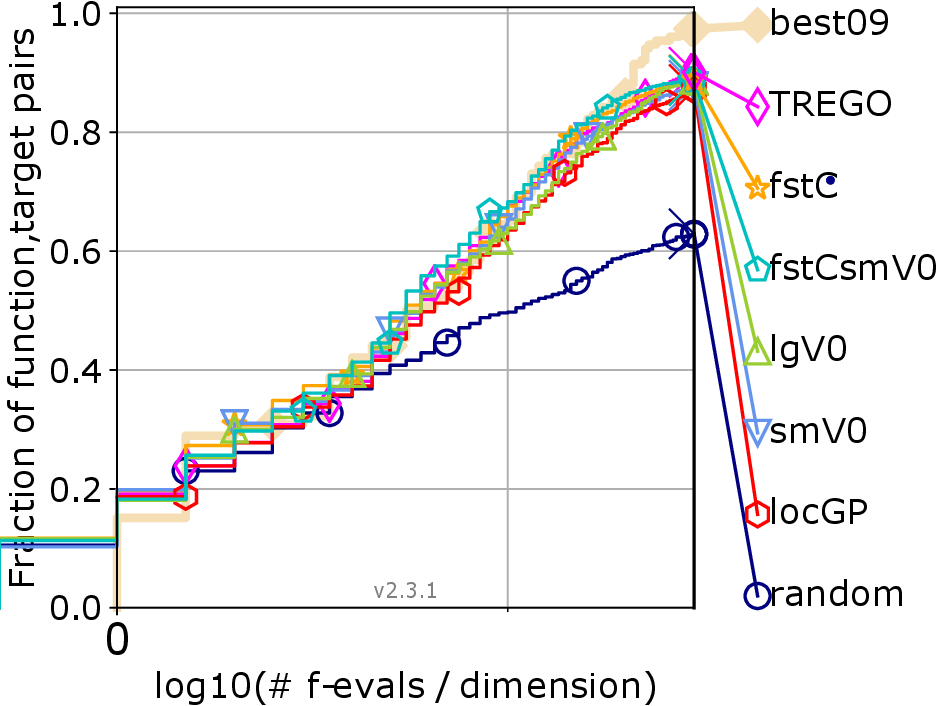}&
\includegraphics[width=0.36\textwidth]{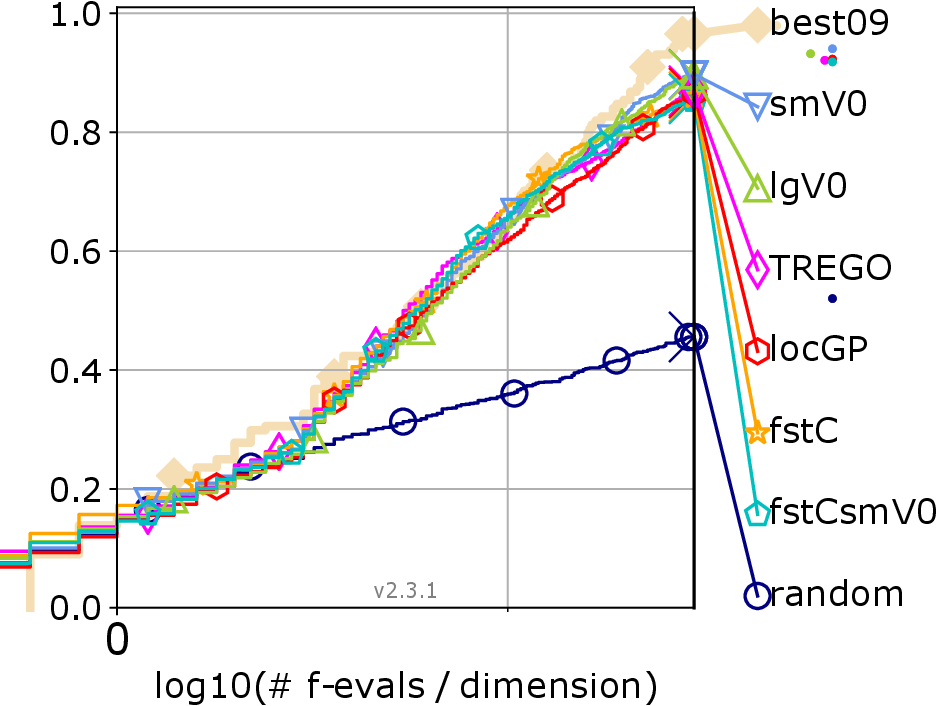}\\
\hline
\end{tabular}
\caption{Effect of changing the amount of local and global iterations (top), and changing the other parameters of the TREGO algorithm (bottom). Performance is reported in terms of ERTD, averaged over the entire noiseless BBOB testbed in 2 (left) and 5 (right) dimensions. Run length is $30\times n$. 
\label{fig-summarymerged}
}
\end{figure}

Figure~\ref{fig-summarymerged}, bottom row,
shows the average performance of other variants of TREGO. 
Overall, TREGO has very little sensitivity to its internal parameters, the average performances of all TREGO variants being similar in both dimensions. 
The robustness of TREGO performance with respect to the other parameters is an advantage of the method, and is in line with what is generally observed for trust region based algorithms.

The effects of the TREGO parameters are studied by function groups in Figure~\ref{fig-CV0locGP_by_groups} (see Appendix). 
The main visible results are:
\begin{itemize}
 \item a slightly positive effect of the local GP (locGP) on the groups 1 and 2 but a strong negative effect on unimodal functions with bad conditioning (group 3), and no effect on the remaining groups. Despite offering attractive flexibility in theory, the local GP provides in practice either limited gain or has a negative impact on performance. As this variant is also more complicated than TREGO, it may be discarded. 
 \item  a positive effect of fast contraction of the trust region (fstC and fstCsmV0) on highly multimodal functions (group 5) during early iterations. By making the trust region more local earlier in the search, the fast contraction allows to reach the easy targets, but this early performance prevents the algorithm from finding other better targets later on (those variants being outperformed by others at the end of the runs).
\end{itemize}
{
The gl1-4 variant of TREGO is shown to offer the best trade-off over all groups between performance and robustness. In our comparison with the state-of-the-art BBO algorithms, we will use the name TREGO to refer to the gl1-4 solver.}

\subsection{Comparison with state-of-the-art BBO algorithms}

Longer runs of length $50n$ (function evaluations) are made with TREGO in dimensions 2, 5 and 10. The results are compared to state-of-the-art Bayesian optimization algorithms: a vanilla EGO, that serves as a baseline, TRIKE (see Section \ref{sec:relatedwork}), SMAC, {  DTS-CMA, Nomad and MCS.} 
A COCO test campaign of such a set of algorithms up to dimension 10, with run length of $50n$ and 15 repetitions of the optimizations takes of the order of 3 weeks of computing time on a recent single processor.

DTS-CMA~\cite{DTS_CMA_2019} is a surrogate-assisted evolution strategy based on a combination of the CMA-ES algorithm and Gaussian process surrogates. 
The DTS-CMA solver is known to be very competitive compared to the state-of-the-art black-box optimization solvers particularly on some classes of multimodal test problems. 
SMAC~\cite{SMAC_2011} (in its BBOB version) is a BO solver that uses an isotropic GP to model the objective function and a stochastic local search to optimize the expected improvement. SMAC is known to perform very well early in the search compared to the state-of-the-art black-box optimizers. { Nomad~\cite{NOMAD2021,NOMAD2011} is a \texttt{C++} solver based on the mesh adaptive direct search method~\cite{CAudet_JEDennis_2006}. 
We have tested Nomad \texttt{version 4.2.0} via its provided Python interface where the variable neighborhood search (VNS) strategy was enabled to enhance its global exploration. 
Nomad enjoys similar convergence properties to those of TREGO, hence a comparison between the two solvers is meaningful. 
MCS~\cite{MCS_huyer} is a multilevel coordinate search solver that balances global and local search (the latter using
quadratic interpolation). MCS is among the best DFO solvers on bound constrained optimization problems~\cite{LRios_NSahinidis_2013}.

DTS-CMA, SMAC and MCS results are directly extracted from the COCO database. This is not the case of Nomad and TRIKE.} As TRIKE follows a relatively standard BO framework, we use our own implementation to compare TREGO against it. As TURBO has a complex structure and the available code is too computationally demanding to be used directly with COCO, it is left out of this study. Figure~\ref{fig-summaryCompareLong} gives the average performance of the algorithms on all the functions of the testbed.
Results in 5 and 10 dimensions split by function groups are provided in Figure~\ref{fig-CompareMultiLong}.

\begin{figure}
\centering
\begin{tabular}{|ccc|}
\hline
$n=2$ & $n=5$ & $n=10$ \\
\hline \hline
\includegraphics[trim=0mm 0mm 0mm 0mm, clip,width=0.3\textwidth,height=3cm]{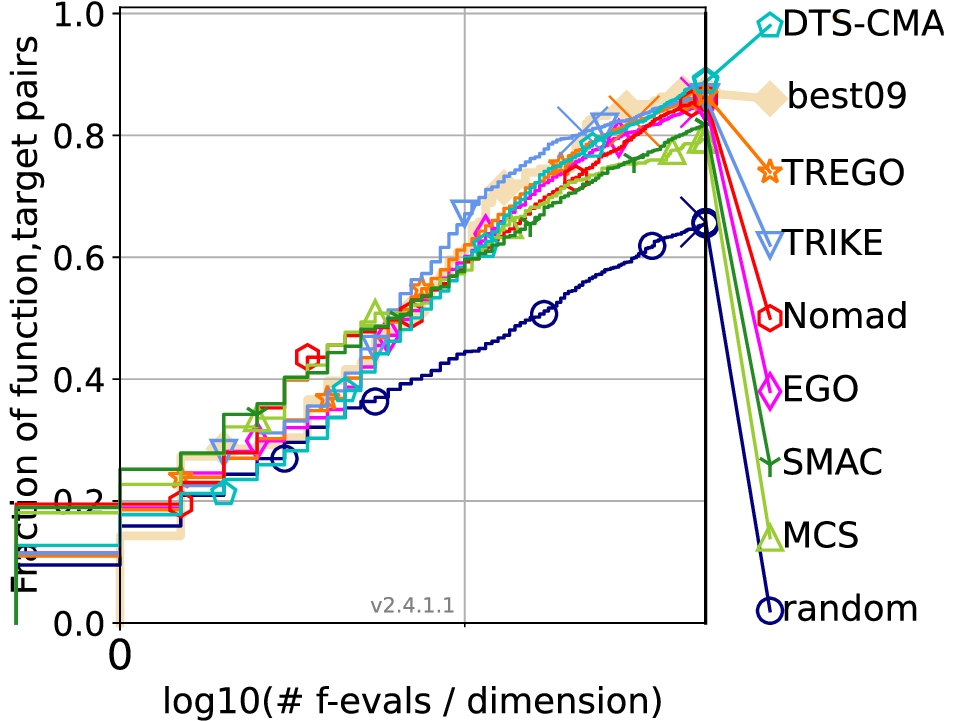}&
\includegraphics[trim=0mm 0mm 0mm 0mm, clip, width=0.3\textwidth,height=3cm]{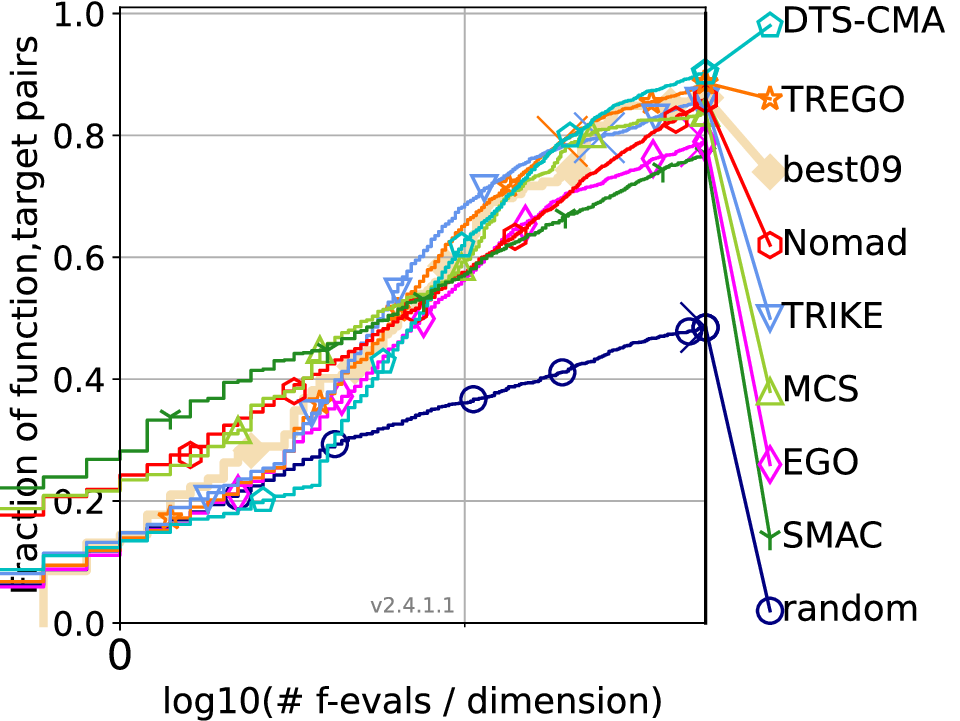}&
\includegraphics[trim=0mm 0mm 0mm 0mm, clip, width=0.3\textwidth,height=3cm]{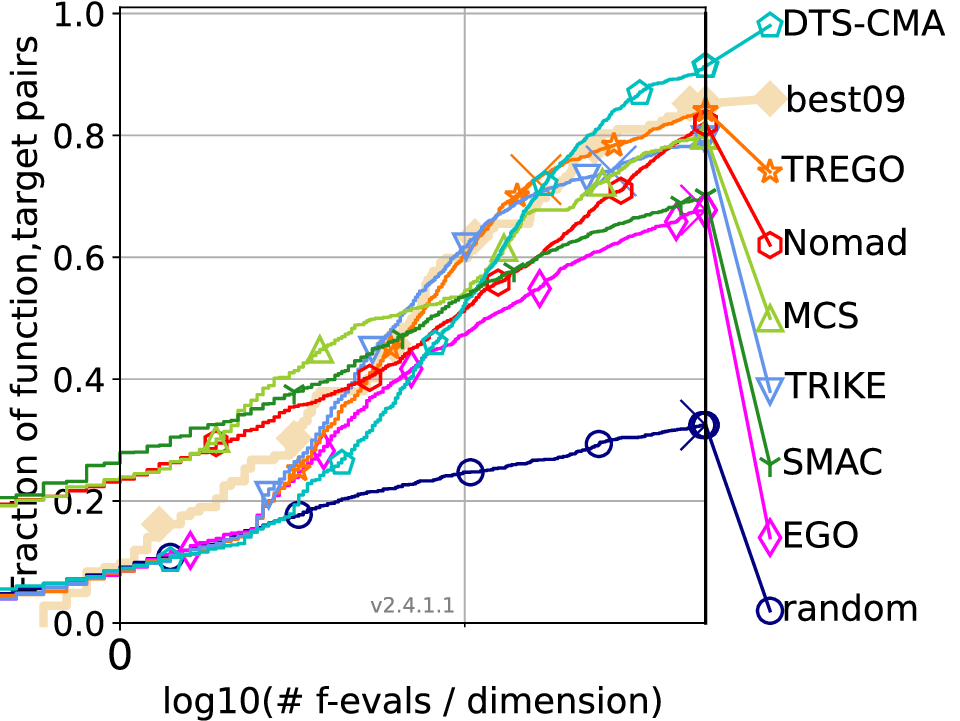} \\
\hline
\end{tabular}
\caption{Comparison of TREGO with state-of-the-art optimization algorithms, averaged over the entire COCO testbed in 2, 5 and 10 dimensions. Run length = $50\times n$.
\label{fig-summaryCompareLong}}
\end{figure}

\begin{figure}
\begin{tabular}{|c||c@{}c@{}c@{}|}
\hline 
& {Multimodal} & {Multimodal, weak struct.} & {Unimodal, low cond.} \\
\hline \hline
\rotatebox[origin=l]{90}{\text{~~~~~~~~$n=5$}}& 
\includegraphics[trim=0mm 0mm 0mm 0mm, clip, width=0.3\textwidth]{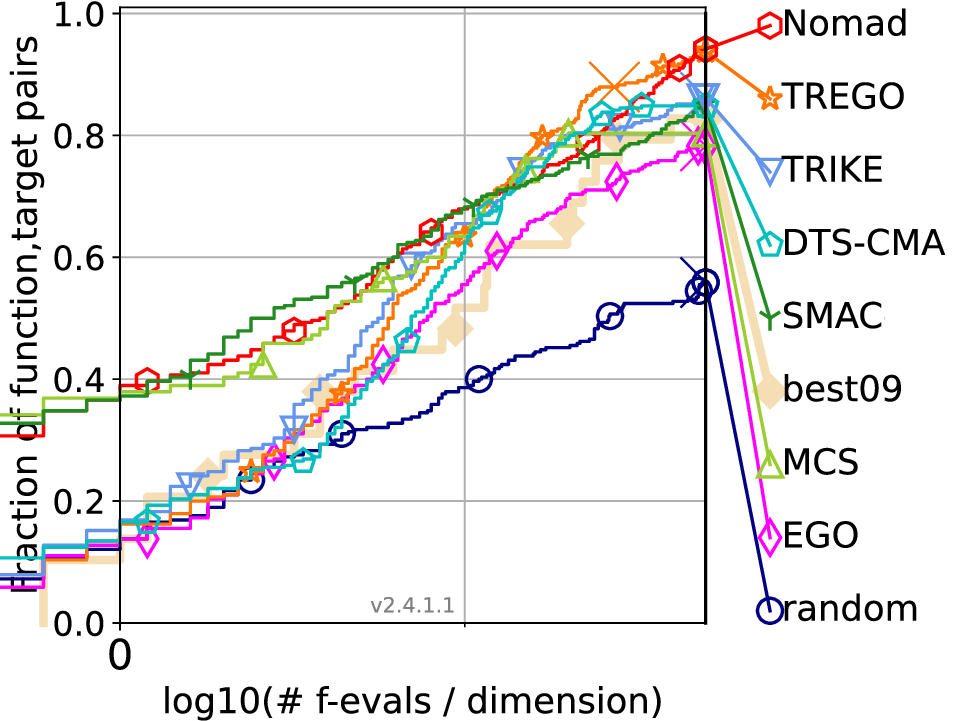} &
\includegraphics[trim=0mm 0mm 0mm 0mm, clip, width=0.3\textwidth]{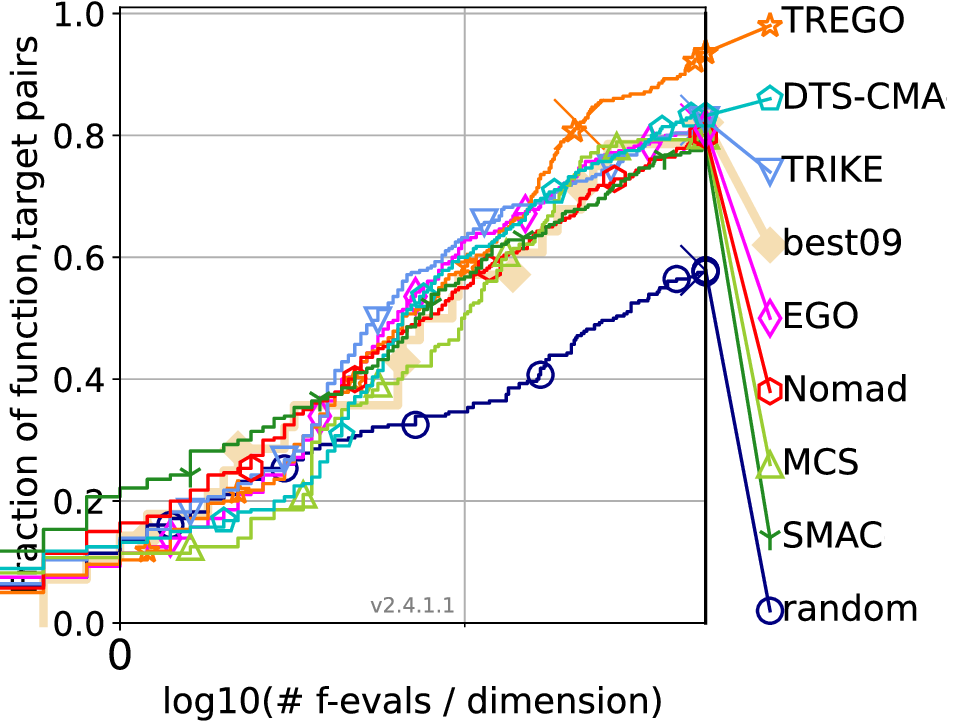}&
\includegraphics[trim=0mm 0mm 0mm 0mm, clip, width=0.3\textwidth]{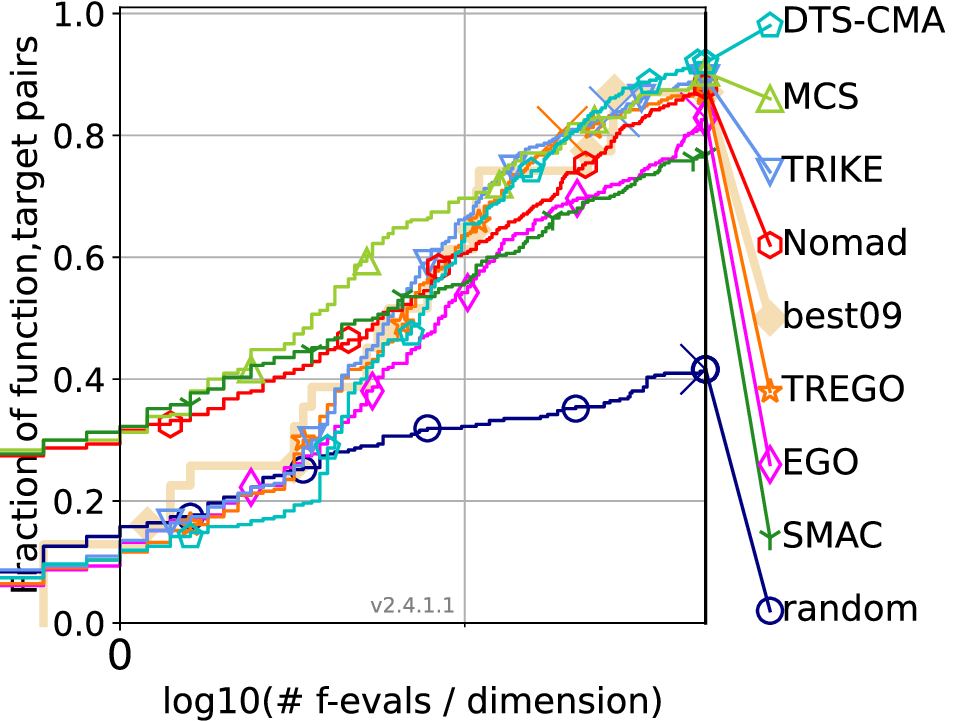} \\ \hline \hline
\rotatebox[origin=l]{90}{\text{~~~~~~~~$n=10$}}& 
\includegraphics[trim=0mm 0mm 0mm 0mm, clip, width=0.3\textwidth]{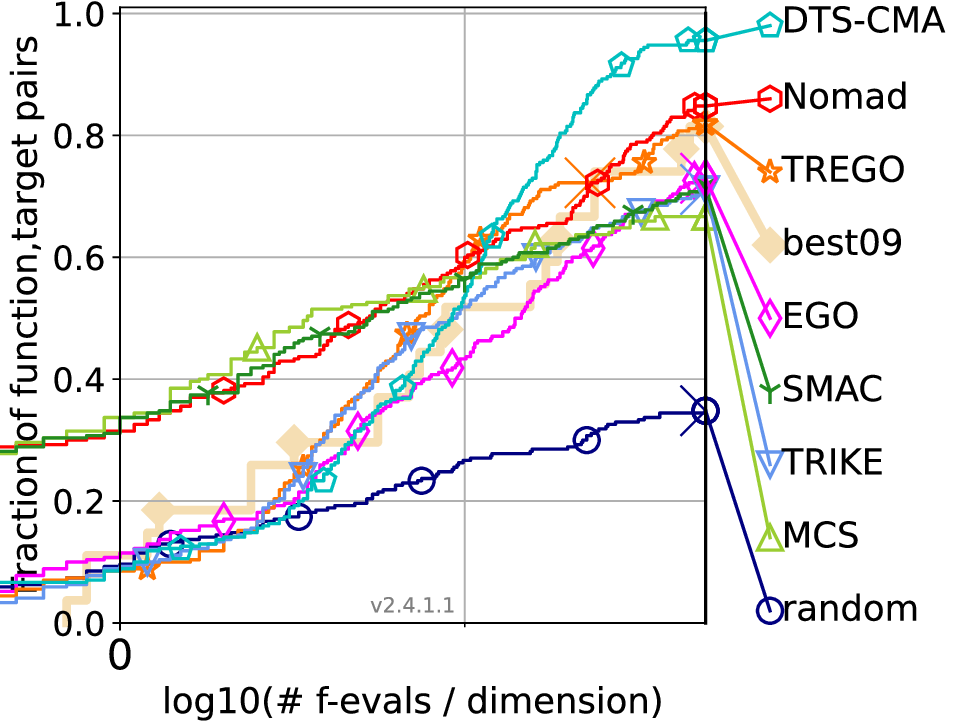}&
\includegraphics[trim=0mm 0mm 0mm 0mm, clip, width=0.3\textwidth]{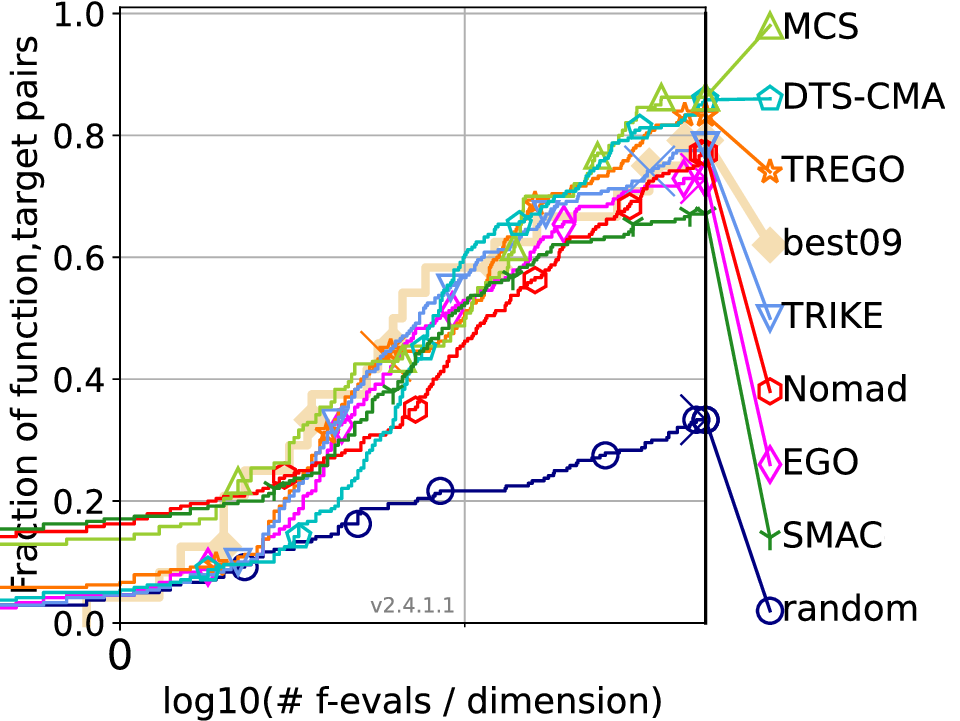}&
\includegraphics[trim=0mm 0mm 0mm 0mm, clip, width=0.3\textwidth]{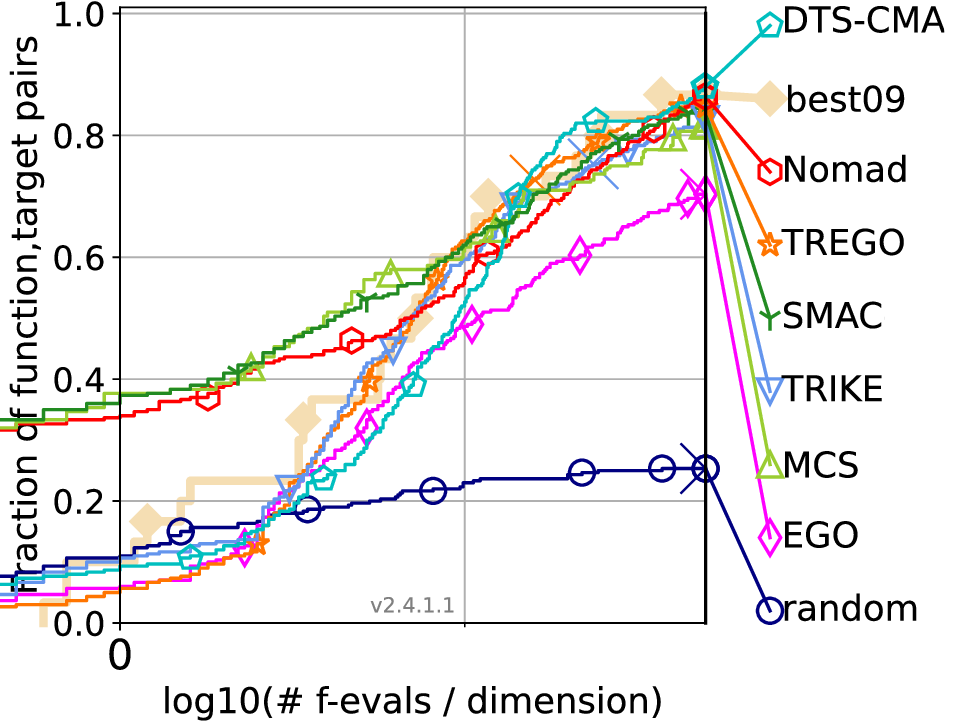} \\ \hline

\end{tabular}
\caption{Comparison of TREGO with state-of-the-art optimization algorithms, averaged over the multi-modal functions with adequate (left, f15 to f19) and weak (middle, f20 to f24) global structure, unimodal functions with low conditioning (right), $n=5$ (top row) and $n=10$ (bottom row) dimensions. Run length = $50\times n$.
Results for the other groups are given in Appendix, Figure~\ref{fig-CompareMultiLong2}.
\label{fig-CompareMultiLong}}
\end{figure}

\paragraph{EGO} is significantly outperformed by both trust regions algorithms (i.e., TREGO and TRIKE). 
This performance gap is limited for $n=2$ but very visible for $n=5$ and even higher for $n=10$. It is also significant for any budget (as soon as the shared initialization is done). The improvement is also visible for all function groups (see Figure~\ref{fig-CompareMultiLong}), in particular for groups with strong structure. For the multimodal with weak structure group, the effect is mostly visible for the larger budgets.

{
\paragraph{Nomad} has an overall performance comparable to TREGO. Nomad is shown in particular to be very efficient for small budgets ($< 20 \times n$) but then it gets outperformed by TREGO as the evaluation budget gets larger. The performance gap between Nomad and TREGO is limited for $n=2$ but very visible for $n=5$ and even higher for $n=10$, see Figure~\ref{fig-summaryCompareLong}. The good start of Nomad can be explained by the fact that it requires only one point to start the optimization process, while Bayesian optimization solvers need a set of points (i.e., the initial DoE) to initiate the optimization process. Thanks to its variable neighborhood search strategy, Nomad seems to outperform most of the tested solvers on the group of multimodal optimization problems, see Figure~\ref{fig-CompareMultiLong}.

\paragraph{MCS} is outperformed by most of the tested solvers despite its very good performance at the early stages of the optimization process. In fact, the performance of MCS at the beginning seems to deteriorate very fast as the the budget is getting larger, particularly when the regarded optimization problems are multimodal, see Figure~\ref{fig-CompareMultiLong}.
}

\paragraph{SMAC} has an early start and is visibly able to start optimizing while all other methods are still creating their initial DoE.
However, it is outperformed by all trust region variants before the number of evaluations reaches 10 times the problem dimension (vertical line on the graphs).
This effect also increases with dimension. 

\paragraph{DTS-CMA} has conversely a slower start, 
so that it is slightly outperformed by trust regions for small budgets ($< 20 \times n$). However, for large budgets and $n=10$, 
DTS-CMA largely outperforms other methods on average. However, looking at Figure~\ref{fig-CompareMultiLong}, DTS-CMA clearly outperforms the other methods (including the best09 baseline) on multimodal functions with strong structure for $n=10$ and large budgets, while TREGO remains competitive in other cases.

\paragraph{TRIKE} has an overall performance comparable to TREGO. 
For $n=5$, it slightly outperforms the other methods for intermediate budget values, but looses its advantage for larger budgets. 
Figure~\ref{fig-CompareMultiLong2} (see Appendix) reveals that this advantage is mainly achieved on the unimodal group with high conditioning, but
 on multi-modal problems, TREGO's ability to perform global steps offer a substantial advantage.

\paragraph{Overall performance} Overall, this benchmark does not reveal a universal winner. {SMAC, Nomad and MCS excel with extremely limited budgets}, while DTS-CMA outperforms the other methods for the largest dimensions and budgets. TREGO is overall very competitive on intermediate values, in particular for multi-modal functions.

\paragraph{Discussion} It appears clearly from our experiments that trust regions are an efficient way to improve EGO's scalability with dimension. EGO is known to over-explore the boundaries in high dimension~\cite{oh2018bock,Turbo_2019}, and narrowing the search space to the vicinity of the current best point naturally solves this issue. {Thus}, since EGO is outperformed for any budget, we can conclude that the gain obtained by focusing early on local optima is not lost later by missing the global optimum region.
Trust regions also improve performance of EGO on problems for which GPs are not the most natural fit (i.e. unimodal functions). 
For this class of problems, the most aggressively local algorithm (TRIKE) can perform best in some cases (Figure~\ref{fig-CompareMultiLong2}), however our more balanced approach is almost as good, if better (Figure~\ref{fig-CompareMultiLong2}, unimodal functions with low conditioning).
On the other hand, maintaining a global search throughout the optimization run allows escaping local optima and ultimately delivering better performance for larger budgets (see in particular Figure~\ref{fig-CompareMultiLong}, all multimodal functions). 
%


\section{Conclusions and perspectives}
\label{sec:conc}
In this work, {the TREGO method is introduced: a Bayesian optimization algorithm based on a trust-region mechanism for the optimization of expensive-to-evaluate black-box functions.}
TREGO builds on the celebrated EGO algorithm by alternating between a standard global step and a local step during which the search is limited to a trust region.
{Equipped with such a local step, TREGO rigorously achieves} global convergence, while enjoying the flexible predictors and efficient exploration-exploitation trade-off provided by the GPs.
{An extensive benchmark is then performed, which allowed us to form the following conclusions}:
\begin{itemize}
 \item TREGO benefits from having a relatively high proportion of local steps, but is otherwise {almost} insensitive to its other parameters.
 \item A more complex approach involving both a local and a global model, which is possible in the TREGO framework, does not provide any benefit.
 \item TREGO significantly outperforms EGO in all tested situations.
 \item TREGO is a highly competitive algorithm for multi-modal functions with moderate dimensions and budgets. 
\end{itemize}
Making TREGO a potential overall winner on the experiments reported here is an avenue for future work.
This would require improving its performance on unimodal functions with high conditioning,
and improving its performance at very early steps, for example by leveraging SMAC for creating the initial DoEs. 
{ Our empirical evaluation focused on bound constrained BBO problems. 
However, TREGO readily applies to the case of explicit, non-relaxable constraints, which may be studied in the future.
Moreover, inspired by e.g.~\cite{CAudet_JEDennis_2009,YDiouane_2021,Gratton_Vicente_2014} from the DFO community and ~\cite{picheny2016bayesian,schonlau1998global} from the BO one, TREGO can also be naturally extended to handle constraints that are allowed to be violated during the optimization process.} Another important future work may include the extension of TREGO to the case of noisy observations, following recent 
results in DFO~\cite{anagnostidis2021direct,Sto_Mads_2019,Sto_trust_region_2018,diouane2022ES} and established BO techniques~\cite{picheny2013benchmark}.

\small
\bibliographystyle{plain}

\newpage
\appendix
\section{Pseudo-code of the TREGO algorithm}
\label{apendix:A}
\LinesNumberedHidden

\begin{algorithm}[!ht]
\footnotesize
	\SetAlgoNlRelativeSize{.1}
	\caption{\bf \bf A Trust-Region framework for EGO (TREGO).}
	\label{alg:TREGO}
	\SetAlgoLined
	\KwData{Create an initial \doe ~{$\mathcal{D}_{t_0}=\{x_1, x_2, \ldots,x_{t_0}\}$} of $t_0$ points in a given set $\Omega \subset \real^n$ with a given method. Set {$\mathcal{Y}_{t_0} = \{ f (x_1), f (x_2), \ldots, f(x_{t_0}) \}$}. Choose $G\ge 0$ the number of the global steps and $L\ge 1$ the number of the local steps. Initialize the step-size parameter $\sigma_0$, $x^*_0 \in \mathcal{D}_{t_0}$, choose {the constants $\beta$, $\gamma$, $d_{\min}$ and $d_{\max}$ such that $0< \beta < 1<\gamma$}  and $0 < d_{\min} < d_{\max}$. Select a forcing function $\rho(.)$ and set $k=0$ and $t=t_0$\;}
	
	\While{some stopping criterion is not satisfied}{
		\tcc{\large A global phase over $\Omega$:}	
		\For{$i=1,\ldots,G$}{
		\textbf{Step 1 (global acquisition function maximization):} 
		
				Set
				\begin{equation*} x^{\glb}_{t} := \underset{x \in \Omega}{\argmax} ~ \alpha(x; \mathcal{D}_t); \end{equation*}
		
		\textbf{Step 2 (update the \doe):}  
		Set $\mathcal{D}_{t+1} = \mathcal{D}_{t} \cup \left\lbrace x^{\glb}_{t} \right\rbrace$ and $\mathcal{Y}_{t+1} = \mathcal{Y}_t \cup \left\lbrace f \left( x^{\glb}_{t} \right) \right\rbrace$\;
		Increment $t$\; 
		}
		Let $x^{\glb}_{k+1}$ be the best point (in term of $f$) in the \doe ~$\mathcal{D}_{t}$\; 
		\textbf{Step 3 (imposing sufficient decrease globally):} 	
		
		\eIf{$f(x_{k+1}^{\glb}) \leq f(x^*_k ) - \rho (\sigma_k)$}{the global phase is successful, set  $x^*_{k+1} = x^{\glb}_{k+1}$ and $\sigma_{k+1}= \gamma \sigma_k$\;  }{
			\tcc{\large A local phase over {the trust-region} $\Omega_k$:}
			\For{$i=1,\ldots,L$}{
			\textbf{Step 4 (local acquisition function maximization):} 
			
					Set
				\begin{equation*}
		x^{\lcl}_{t} :=\underset{x \in \Omega_k}{\argmax} ~ \alpha(x; \mathcal{D}_t),
		\end{equation*}
		where  {$\Omega_k$ is the trust-region given by} $\Omega_k = \{x \in \Omega\; | \;  d_{\min} \sigma_k \le \| x - x^*_k \| \le d_{\max} \sigma_k\}$\;
		
		\textbf{Step 5 (update the \doe):} 
	        Set $\mathcal{D}_{t+1} = \mathcal{D}_{t} \cup \left\lbrace x^{\lcl}_{t} \right\rbrace$ and $\mathcal{Y}_{t+1} = \mathcal{Y}_t \cup \left\lbrace f \left( x^{\lcl}_{t} \right) \right\rbrace$\;
		Increment $t$\; 
		}
		Let $x^{\lcl}_{k+1}$ be the best point (in term of $f$) in the \doe ~$\mathcal{D}_{t}$\; 
		\textbf{Step 6 (imposing sufficient decrease locally):} 
		
		\eIf{$f(x^{\lcl}_{k+1}) \leq f(x^*_k ) - \rho (\sigma_k)$}{the local phase and iteration are successful, set  $x^*_{k+1} = x^{\lcl}_{k+1}$ and {$\sigma_{k+1}= \gamma \sigma_k$} \;}{the local phase and iteration are not successful, set $x^*_{k+1}=x^*_{k}$, and $\sigma_{k+1} = {\beta}  \sigma_k$\;}}
		  
		  Increment $k$\;
	}
\end{algorithm}

\newpage
\section{Functions of the BBOB noiseless testbed}
\label{apendix:B}
\begin{table}[!ht]
\footnotesize
\begin{tabular}{|c|p{0.3\textwidth}p{0.65\textwidth}|}
\hline
ID & name & comments \\
\cline{1-3}
\multicolumn{3}{c}{separable functions}\\
\cline{1-3}
f1 & Sphere & unimodal, allows to checks numerical accuracy at convergence \\
f2 & Ellipsoidal & unimodal, conditioning $\approx 10^6$ \\
f3 & Rastrigin & $10^n$ local minima, spherical global structure \\
f4 & B\"uche-Rastrigin & $10^n$ local minima, asymmetric global structure \\
f5 & Linear Slope & linear, solution on the domain boundary \\
\cline{1-3}
\multicolumn{3}{c}{functions with low or moderate conditioning}\\
\cline{1-3}
f6 & Attractive Sector & unimodal, highly asymmetric \\
f7 & Step Ellipsoidal & unimodal, conditioning $\approx 100$, made of many plateaus \\
f8 & Original Rosenbrock & good points form a curved $n-1$ dimensional valley \\
f9 & Rotated Rosenbrock & rotated f8 \\
\cline{1-3}
\multicolumn{3}{c}{unimodal functions with high conditioning $\approx 10^6$ }\\
\cline{1-3}
f10 & Ellipsoidal & rotated f2 \\
f11 & Discus & a direction is 1000 times more sensitive than the others \\
f12 & Bent Cigar & non-quadratic optimal valley \\
f13 & Sharp Ridge & resembles f12 with a non-differentiable bottom of valley \\
f14 & Different Powers & different sensitivities w.r.t. the $x_i$'s near the optimum \\
\cline{1-3}
\multicolumn{3}{c}{multimodal functions with adequate global structure}\\
\cline{1-3}
f15 & Rastrigin & rotated and asymmetric f3 \\
f16 & Weierstrass & highly rugged and moderately repetitive landscape, non unique optimum \\
f17 & Schaffers F7 & highly multimodal with spatial variation of frequency and amplitude, smoother and more repetitive than f16 \\
f18 & moderately ill-conditioned Schaffers F7 & f17 with conditioning $\approx 1000$ \\
f19 & Composite Griewank-Rosenbrock & highly multimodal version of Rosenbrock \\
\cline{1-3}
\multicolumn{3}{c}{multimodal functions with weak global structure}\\
\cline{1-3}
f20  & Schwefel & $2^n$ most prominent optima close to the corners of a shrinked and rotated rectangle \\
f21 & Gallagher's Gaussian 101-me peaks & 101 optima with random positions and heights, conditioning $\approx 30$ \\
f22 & Gallagher's Gaussian 21-hi peaks & 21 optima with random positions and heights, conditioning $\approx 1000$ \\
f23 & Katsuura & highly rugged and repetitive function with more than $10^n$ global optima \\
f24 & Lunacek bi-Rastrigin & highly multimodal function with 2 funnels, one leading to a local optimum and covering about 70\% of the search space \\
\hline
\end{tabular}
\caption{
Functions of the BBOB noiseless testbed, divided in groups.
\label{tab-bbob_functions}
}
\end{table}

\newpage
\section{Complementary experimental results}
\begin{figure}[!ht]
\centering
\begin{tabular}{|c||c@{}ccc@{}|}
\hline
& \text{Local / global ratio } & & &\text{Other parameters }\\
\hline \hline
\rotatebox[origin=l]{90}{\text{~~~~~Separable~~~~~~}}& 
\includegraphics[trim=0mm 0mm 10mm 0mm, clip, width=0.33\textwidth,height=3cm]{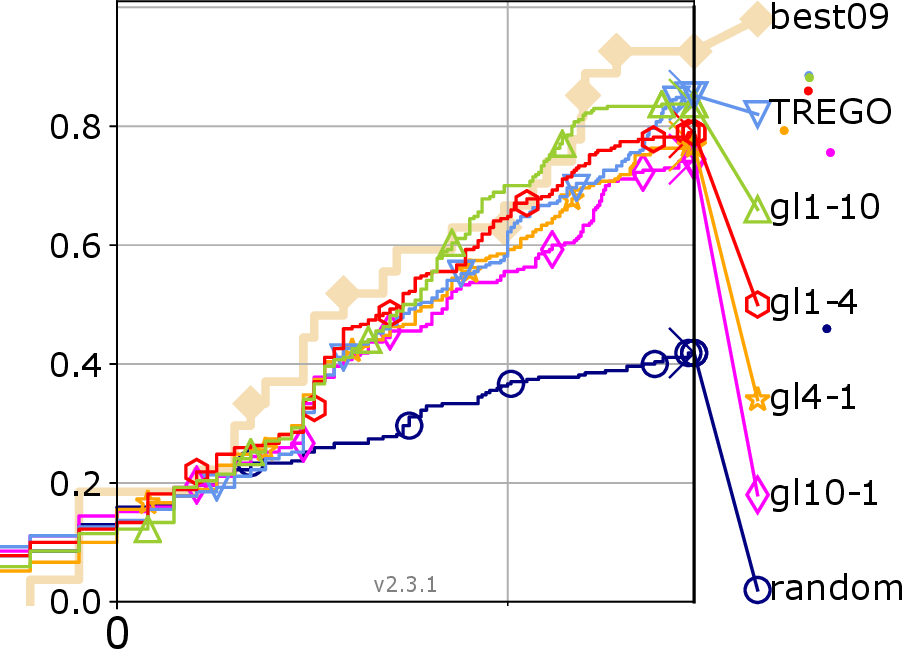} & & &
\includegraphics[trim=0mm 0mm 10mm 0mm, clip, width=0.33\textwidth,height=3cm]{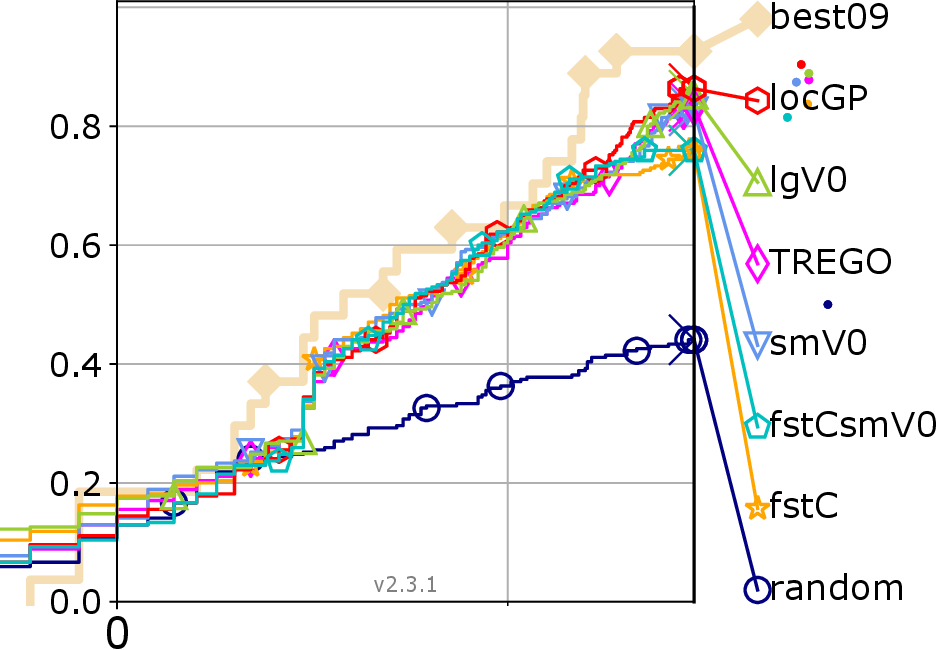} \\\hline\hline
\rotatebox[origin=l]{90}{\text{Low conditioning~~}}& 
\includegraphics[trim=0mm 0mm 10mm 0mm, clip, width=0.33\textwidth,height=3cm]{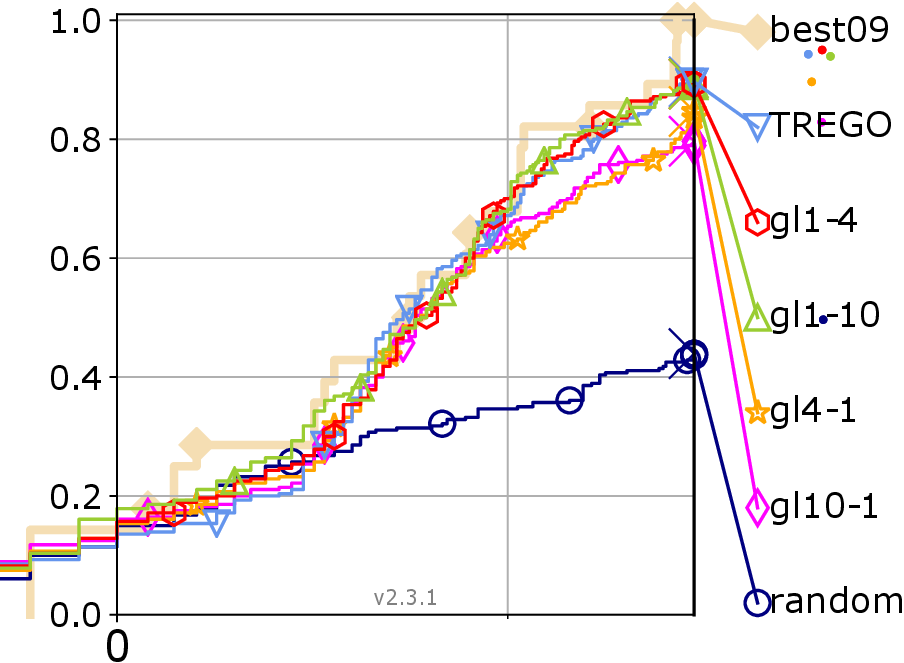}& & &
\includegraphics[trim=0mm 0mm 10mm 0mm, clip, width=0.33\textwidth,height=3cm]{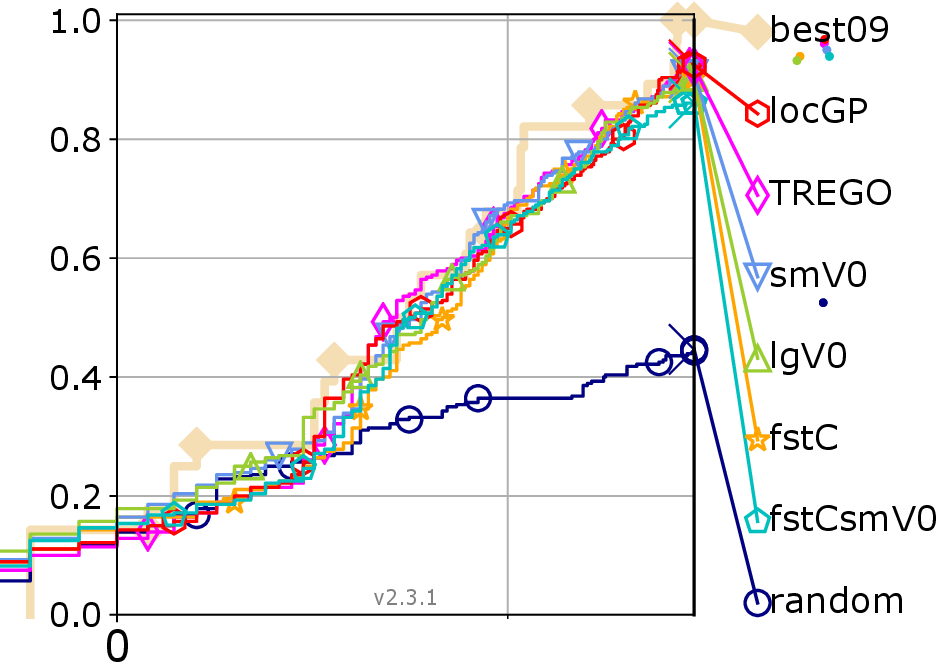}\\\hline\hline
\rotatebox[origin=l]{90}{\text{High conditioning~~}}& 
\includegraphics[trim=0mm 0mm 10mm 0mm, clip, width=0.33\textwidth,height=3cm]{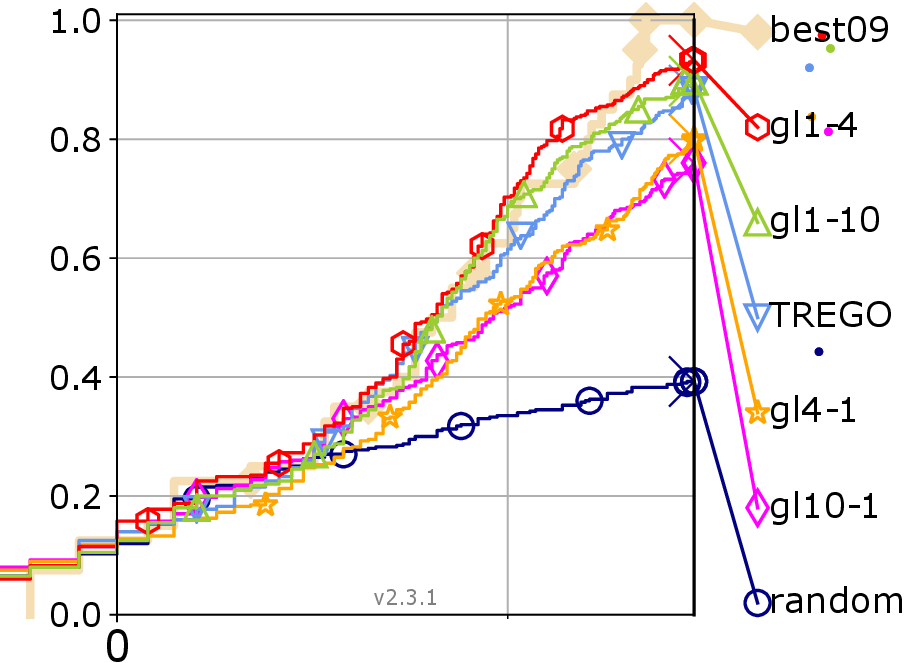} & & &
\includegraphics[trim=0mm 0mm 10mm 0mm, clip, width=0.33\textwidth,height=3cm]{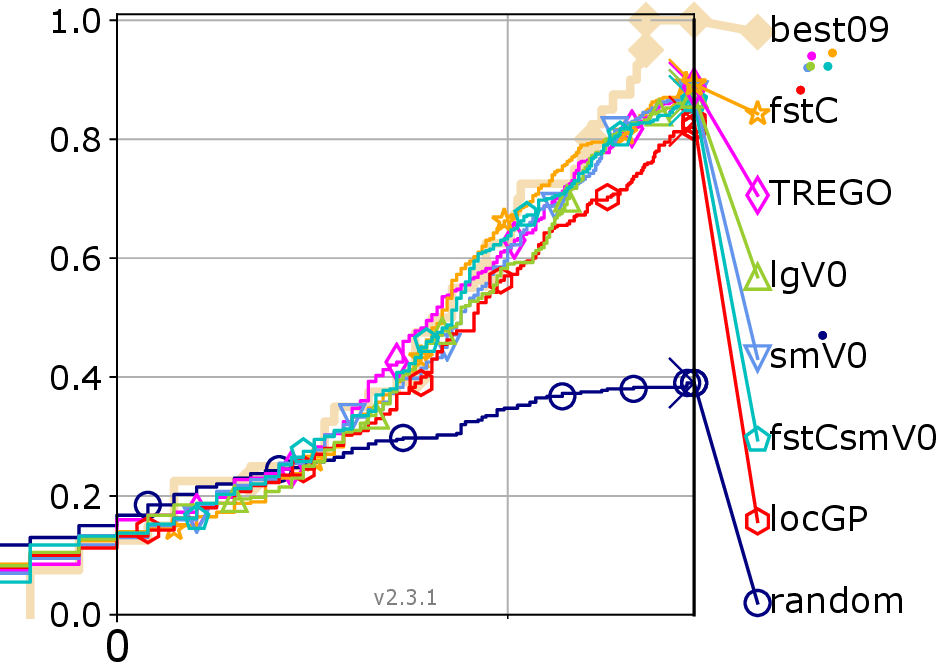} \\\hline\hline
\rotatebox[origin=l]{90}{\text{\small{Multimod., strong struct.}}}& 
\includegraphics[trim=0mm 0mm 10mm 0mm, clip, width=0.33\textwidth,height=3cm]{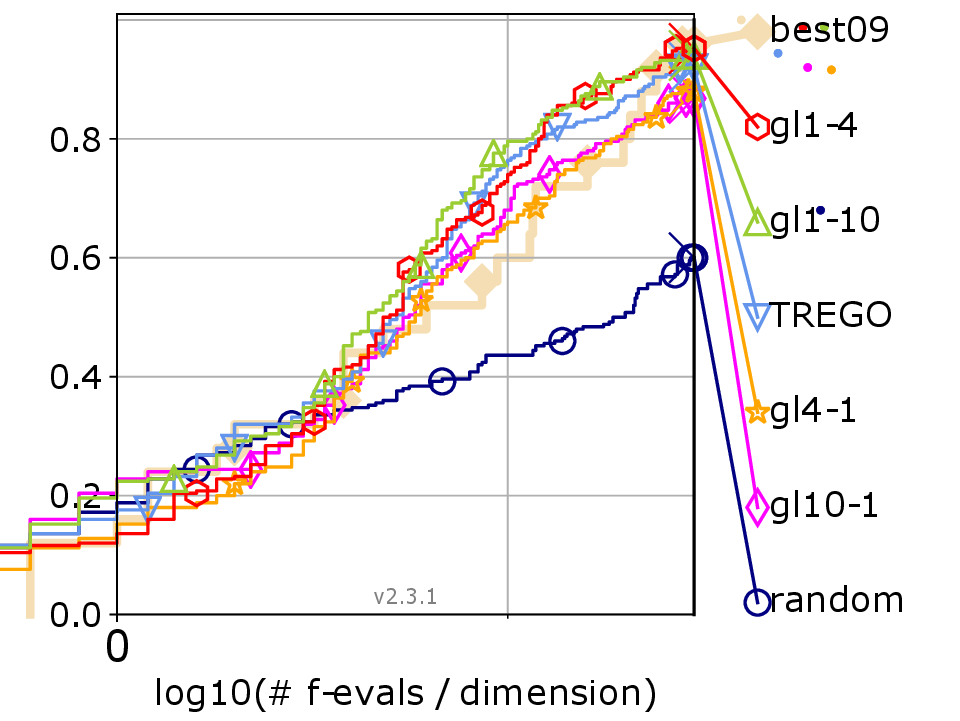}& & &
\includegraphics[trim=0mm 0mm 10mm 0mm, clip, width=0.33\textwidth,height=3cm]{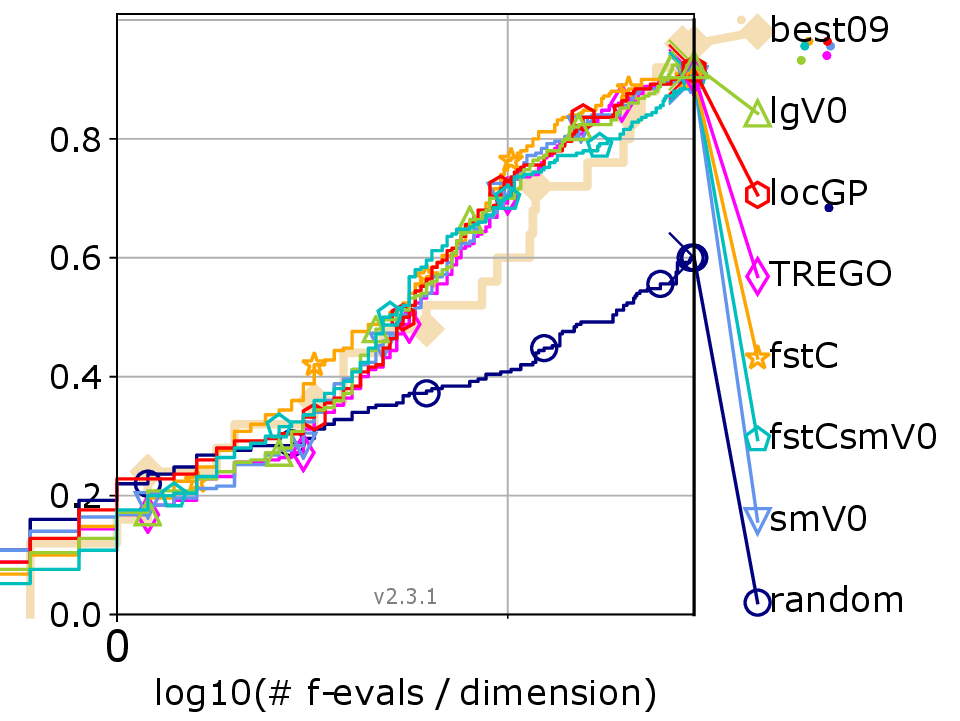}\\\hline\hline
\rotatebox[origin=l]{90}{\text{\small{Multimod., weak struct.}}}& 
\includegraphics[trim=0mm 0mm 10mm 0mm, clip, width=0.33\textwidth,height=3cm]{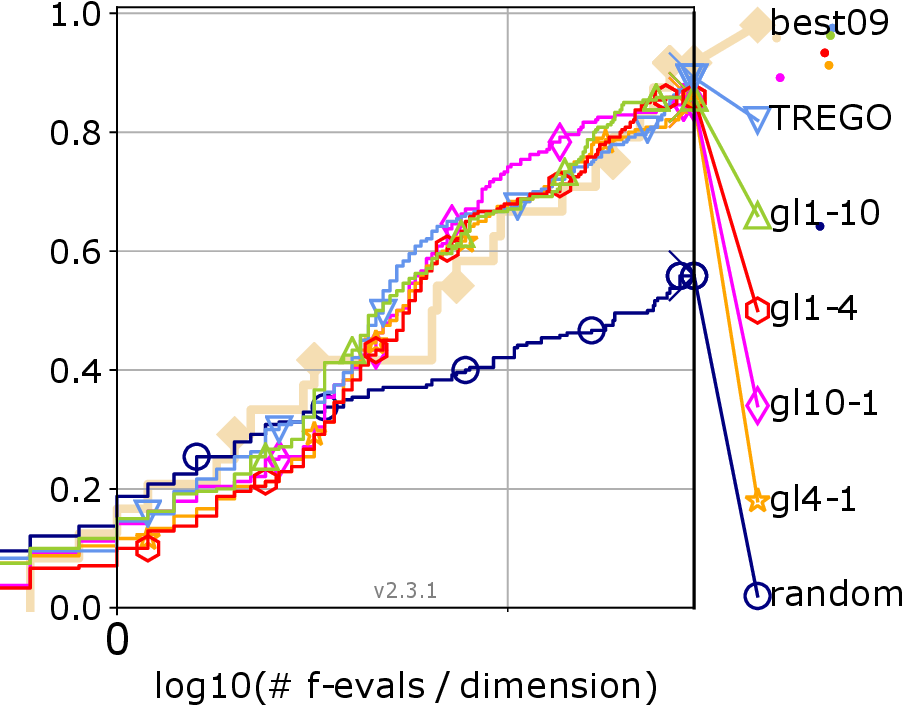} & & &
\includegraphics[trim=0mm 0mm 10mm 0mm, clip, width=0.33\textwidth,height=3cm]{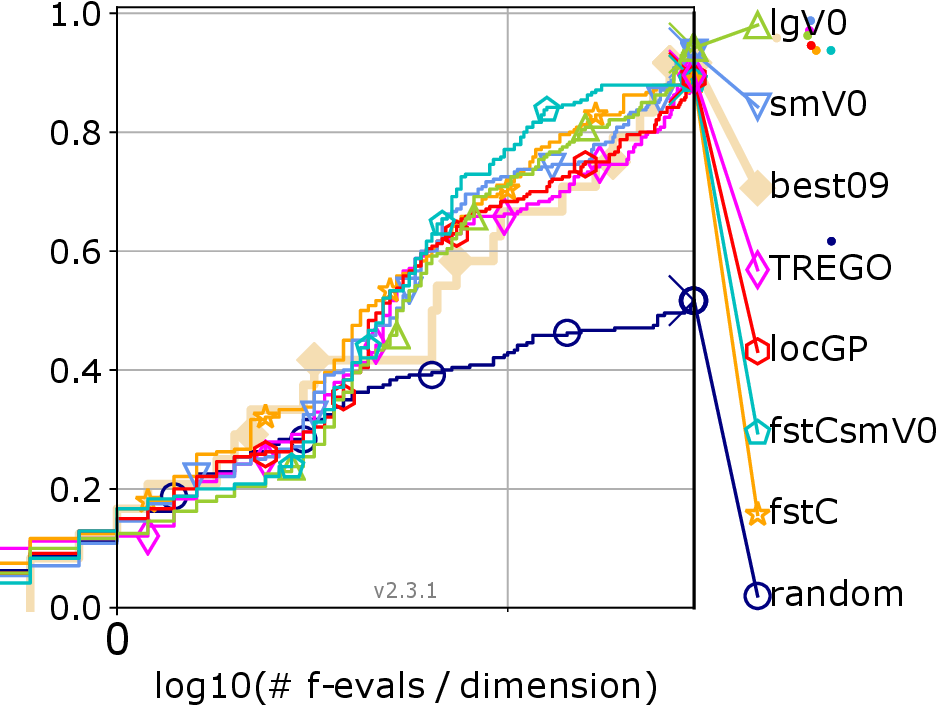}\\
\hline
\end{tabular}
\caption{
Effect of changing parameters of the TREGO algorithm, averaged by function groups for $n=5$. Run length is $30\times n$.
\label{fig-CV0locGP_by_groups}
}
\end{figure}

\begin{figure}
\centering
\begin{tabular}{|c||c@{}ccc@{}|}
\hline
& {Separable} && & {Unimodal, high cond.} \\ \hline \hline
\rotatebox[origin=l]{90}{\text{~~~~~~~$n=5$}}& 
\includegraphics[trim=0mm 0mm 0mm 0mm, clip,  width=0.33\textwidth,height=3cm]{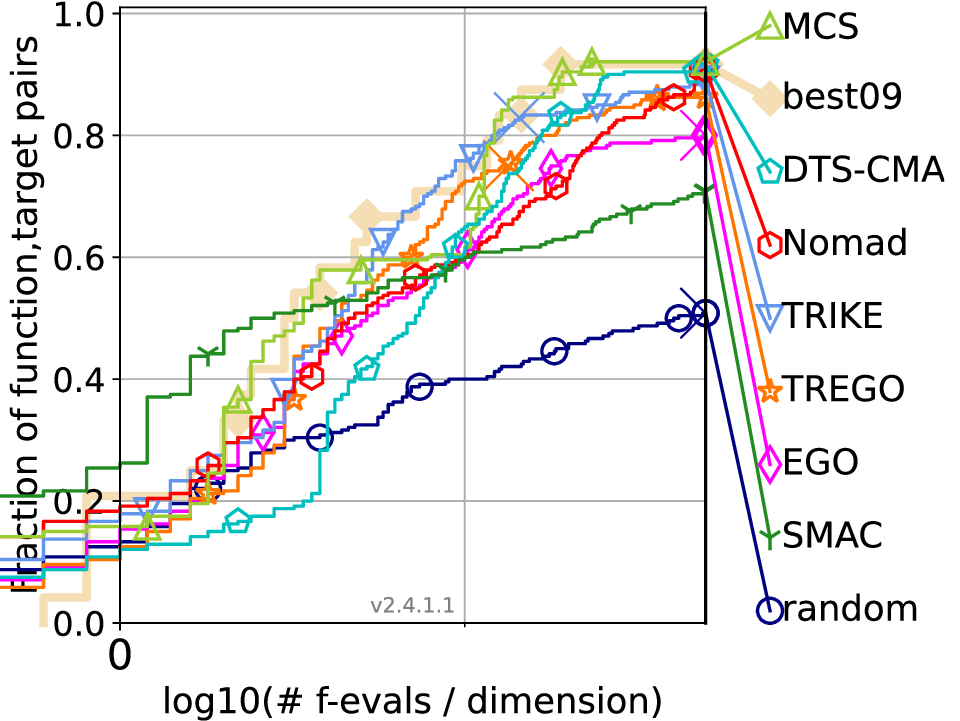} & & &
\includegraphics[trim=0mm 0mm 0mm 0mm, clip,  width=0.33\textwidth,height=3cm]{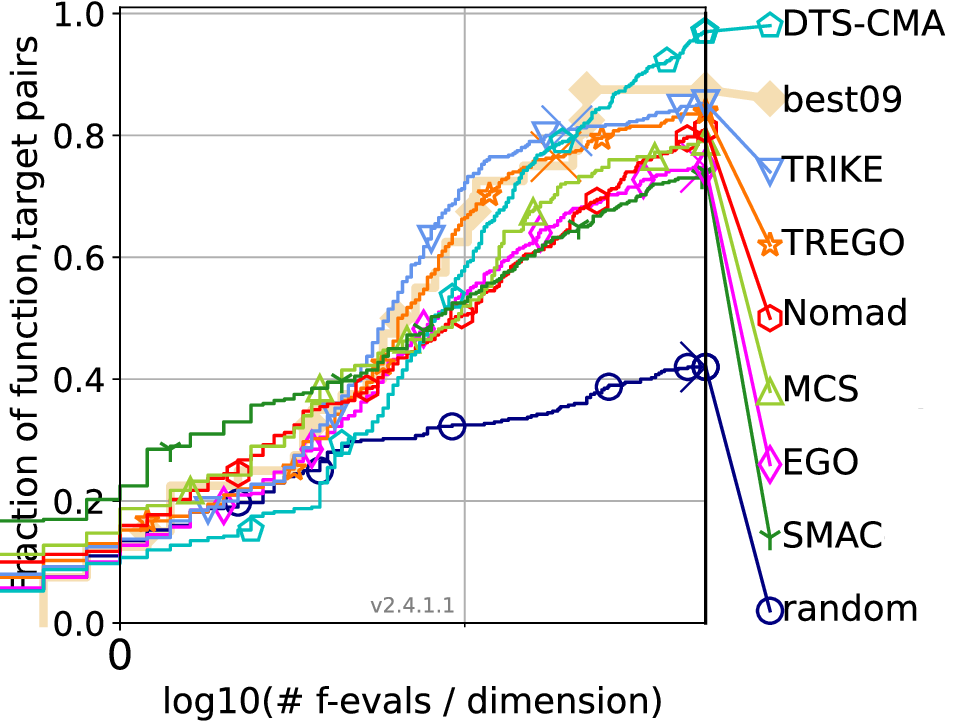} \\ \hline \hline
\rotatebox[origin=l]{90}{\text{~~~~~~~$n=10$}}& 
\includegraphics[trim=0mm 0mm 0mm 0mm, clip,  width=0.33\textwidth,height=3cm]{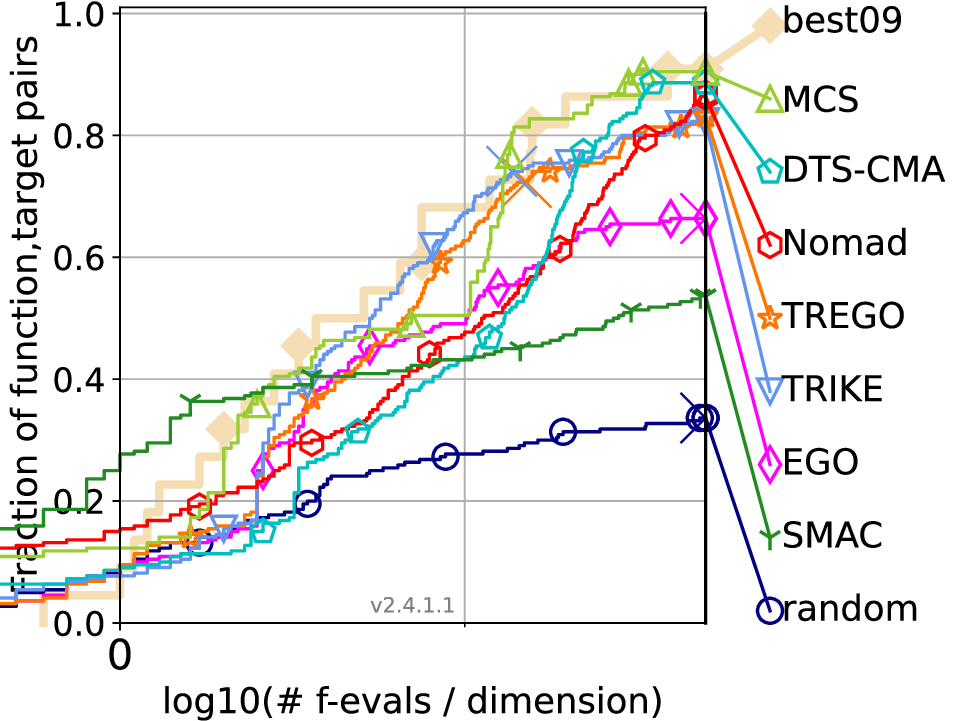}& & &
\includegraphics[trim=0mm 0mm 0mm 0mm, clip,  width=0.33\textwidth,height=3cm]{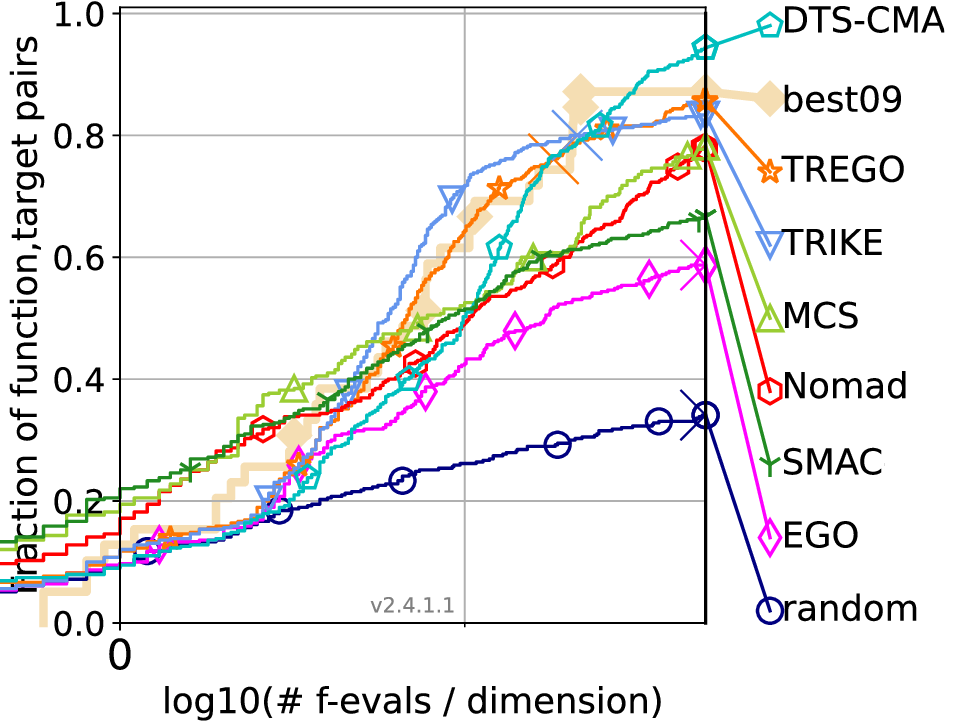}\\
\hline
\end{tabular}
\caption{Comparison of TREGO with state-of-the-art optimization algorithms on separable (left) and unimodal with high conditioning functions (right), for $n=5$ (top) and $n=10$ (bottom). Run length = $50\times n$.
\label{fig-CompareMultiLong2}}
\end{figure}

\end{document}